\documentclass[10pt,reqno]{amsart}
\usepackage{amssymb}
\usepackage{amsmath}
\usepackage{amsthm}
\usepackage{graphicx}

\usepackage{subfigure}

\usepackage{caption}
\usepackage{bm}
\usepackage[colorlinks=true,urlcolor=blue,
citecolor=red,linkcolor=blue,linktocpage,pdfpagelabels,
bookmarksnumbered,bookmarksopen]{hyperref}
\usepackage[titletoc,title]{appendix}
\numberwithin{equation}{section} 
\newcounter{cont}[section] 

\newtheorem{thm}[cont]{Theorem}
\newtheorem{prop}[cont]{Proposition}
\newtheorem{lem}[cont]{Lemma}

\theoremstyle{definition}
\newtheorem{defn}[cont]{Definition}
 \theoremstyle{remark}
 \newtheorem{rem}[cont]{Remark}

\newcommand{\R}{\mathbb{R}}
\newcommand{\e}{\varepsilon}

\newcommand{\N}{\mathbb{N}}

\begin{document}

\title[On a generalized Cahn--Hilliard model with $p$-Laplacian]{On a generalized Cahn--Hilliard model with $p$-Laplacian}

\author[R. Folino]{Raffaele Folino}

\address[R. Folino]{Departamento de Matem\'aticas y Mec\'anica\\Instituto de 
Investigaciones en Matem\'aticas Aplicadas y en Sistemas\\Universidad Nacional Aut\'onoma de 
M\'exico\\Circuito Escolar s/n, Ciudad Universitaria C.P. 04510 Cd. Mx. (Mexico)}

\email{folino@mym.iimas.unam.mx}

\author[L. F. L\'{o}pez R\'{\i}os]{Luis F. L\'{o}pez R\'{\i}os}
\address[L. F. L\'{o}pez R\'{\i}os]{Departamento de Matem\'aticas y Mec\'anica\\Instituto de Investigaciones en Matem\'aticas Aplicadas y 
en Sistemas\\Universidad Nacional Aut\'{o}noma de M\'{e}xico\\Circuito Escolar s/n, Ciudad Universitaria, C.P. 04510\\Cd. de M\'{e}xico (Mexico)}
\email{luis.lopez@iimas.unam.mx}

\author[M. Strani]{Marta Strani}

\address[M. Strani]{Dipartimento di Scienze Molecolari e Nanosistemi\\
Universit\`a Ca' Foscari Venezia Mestre\\
Campus Scientifico\\Via Torino 155, 30170 Venezia Mestre (Italy)}

\email{marta.strani@unive.it}


\subjclass[2010]{}

\maketitle


\begin{abstract} 
A generalized Cahn--Hilliard model in a bounded interval of the real line with no-flux boundary conditions is considered.
The label ``generalized'' refers to the fact that we consider a concentration dependent mobility,
the $p$-Laplace operator with $p>1$ and a double well potential of the form $F(u)=\frac{1}{2\theta}|1-u^2|^\theta$, with $\theta>1$;
these terms replace, respectively, the constant mobility, the linear Laplace operator and the $C^2$ potential satisfying $F''(\pm1)>0$,
which are typical of the standard Cahn--Hilliard model.
After investigating the associated stationary problem and highlighting the differences with the standard results, 
we focus the attention on the long time dynamics of solutions when $\theta\geq p>1$.
In the \emph{critical} case $\theta=p>1$, we prove \emph{exponentially slow motion} of profiles with a transition layer structure,
thus extending the well know results of the standard model, where $\theta=p=2$;
conversely, in the \emph{supercritical} case $\theta>p>1$, we prove \emph{algebraic slow motion} of layered profiles.
\end{abstract}


\section{Introduction}\label{sec:intro}
\subsection{Derivation of the model and motivations}
The celebrated \emph{Cahn-Hilliard equation} in the one-dimensional case reads as
\begin{equation}\label{eq:Ca-Hi}
	u_t=(-\e^2 u_{xx}+F'(u))_{xx},  \qquad \qquad x\in(a,b), \;\, t>0,
\end{equation} 
where $\e>0$ is a small coefficient and $F:\R\to\R$ is a double well potential with wells of equal depth, usually given by 
\begin{equation}\label{eq:classicF}
	F(u)=\frac14|1-u^2|^2.
\end{equation}
This model was originally introduced in \cite{Cahn,Cahn-Hilliard} to model phase separation in a binary system at a fixed temperature and
with constant total density, where $u$ stands for the concentration of one of the two components.
Generally, equation \eqref{eq:Ca-Hi} is considered with homogeneous Neumann boundary conditions
\begin{equation}\label{eq:Neu}
	u_x(a,t)=u_x(b,t)=u_{xxx}(a,t)=u_{xxx}(b,t)=0, \qquad \qquad t>0,
\end{equation}
which are physically relevant as they guarantee that the total mass of the solution is conserved. 
It is well-known that the model \eqref{eq:Ca-Hi}-\eqref{eq:Neu} can be derived as the gradient flow in the zero-mean subspace of the dual of $H^1(a,b)$
of the Ginzburg-Landau energy functional \cite{Fife}
\begin{equation}\label{GL}
	E[u]=\int_a^b \left[ \frac{\e^2}{2} u_x^2+ F(u) \right] \, dx,
\end{equation}
and that the only {\it stable} stationary solutions to \eqref{eq:Ca-Hi}-\eqref{eq:Neu} are minimizers of  the energy $E[u]$ \cite{Zheng}. 
Therefore, solutions to \eqref{eq:Ca-Hi}-\eqref{eq:Neu} converge, as $t \to \infty$, to a limit which has at most a single transition inside the interval $[a,b]$, see \cite{CarrGurtSlem}.
However, in the pioneering works \cite{AlikBateFusc91,Bates-Xun1,Bates-Xun2,Bron-Hilh}, it has been proved that if the initial profile has an $N$-transition layer structure, 
oscillating between the two minimal points $\pm1$ of $F$, then the solution maintains these $N$ transitions for a very long time, 
i.e. a time $T_\e=\mathcal{O}\left(\exp(c/\e)\right)$, as $\e\to0^+$.
In particular, the positive constant $c$ does not depend on $\e$, but only on $F''(\pm1)>0$ and on the distance between the layers.
Hence, we have an example of {\bf metastable dynamics}.

The main goal of this paper is to investigate the metastable properties of the solutions to the following more general version of \eqref{eq:Ca-Hi}, 
named the \emph{generalized Cahn--Hilliard equation} 
\begin{equation}\label{eq:CH-model}
	u_t=\left[D(u)\left(-\varepsilon^p(|u_x|^{p-2}u_x)_x+F'(u)\right)_x\right]_x,
\end{equation}
where $D:\R\to\R^+$ is a strictly positive function, $p>1$ and the function $F:\R\to\R$ is a double well potential with wells of equal depth in $u=\pm1$,
which generalizes the function defined in \eqref{eq:classicF}:
\begin{equation}\label{eq:F}
	F(u)=\frac{1}{2\theta}|1-u^2|^{\theta}, \qquad \theta>1.
\end{equation}
As a consequence, the last term appearing in \eqref{eq:CH-model} is given by
\begin{equation*}
	F'(u)=-u(1-u^2)|1-u^2|^{\theta-2}, \qquad \theta>1.
\end{equation*}
We call \eqref{eq:CH-model} \emph{generalized Cahn--Hilliard equation} because the classic Cahn--Hilliard equation \eqref{eq:Ca-Hi}, with $F$ defined in \eqref{eq:classicF}, 
can be obtained from \eqref{eq:CH-model} by choosing $D\equiv1$, $p=2$ and $\theta=2$ in \eqref{eq:F}.
On the other hand, equation \eqref{eq:CH-model} is a particular case of an even more general Cahn--Hilliard model introduced by Gurtin in \cite{Gurtin}, that in the one-dimensional case reads as
\begin{equation}\label{eq:Gurtin-model}
	u_t=\left\{D\left[-[\partial_v \hat{\Psi}(u,u_x)]_x+\partial_u\hat{\Psi}(u,u_x)-\gamma\right]_x\right\}_x+m,
\end{equation}
where $D$ is a non constant \emph{mobility} (which may depends on $u$ and its derivatives),  $\hat{\Psi}:\R^2\to\R$ is the so-called \emph{free energy}, $\gamma$ is an \emph{external microforce} and $m$ is an \emph{external mass supply}, for further details see \cite{Gurtin}.
In particular, the standard Cahn--Hilliard equation \eqref{eq:Ca-Hi} corresponds to \eqref{eq:Gurtin-model}, with the choices $D\equiv1$, $\gamma\equiv m\equiv0$ and the free energy
$$\hat{\Psi}(u,v):=\frac{\e^2}{2}v^2+F(u)=\frac{\e^2}{2}v^2+\frac{1}{4}|1-u^2|^2.$$
In the model \eqref{eq:CH-model} studied in this article,
$D$ is a concentration dependent mobility (cfr. \cite{Ca-El-NC} and references therein), $\gamma=m=0$ as in the standard case and the free energy is given by
\begin{equation}\label{eq:free-energy}
	\hat{\Psi}(u,v):=\frac{\e^p}{p}|v|^p+F(u)=\frac{\e^p}{p}|v|^p+\frac{1}{2\theta}|1-u^2|^{\theta},\qquad p,\theta>1.
\end{equation}
Notice that the free energy in the standard case corresponds to the particular choice $p=\theta=2$ in \eqref{eq:free-energy}.
Therefore, the model \eqref{eq:CH-model} generalizes the classical one \eqref{eq:Ca-Hi} for three reasons:
\begin{enumerate}
\item First, we consider a concentration dependent mobility instead of the constant one. 
Actually, it is worth to mention that a concentration dependent mobility appears in the original derivation of the Cahn--Hilliard model \cite{Cahn,Cahn-Hilliard}.
Particularly, in the physics literature, there exist one-dimensional, phase-transitional models with concentration dependent, 
strictly positive diffusivities such as the experimental exponential diffusion function for metal alloys (cf. Wagner \cite{Wagner}) 
and the Mullins diffusion model for thermal grooving, $D(u)=D_0/(1+u^2)$, for $D_0>0$, see \cite{Broad,Mullins}. 
\item Second, we consider the $p$-Laplace operator instead of the classic linear diffusion. 
Historically, the $p$-Laplacian first appeared from a power law alternative to Darcy's phenomenological law to describe fluid flow through porous media
(see, for instance, the recent review paper \cite{Ben et al} and the references therein). 
Since then, the $p$-Laplacian has established itself as a fundamental quasilinear elliptic operator and has been intensely studied in the last fifty years.
Up to our knowledge, the effects of the $p$-Laplacian in the Cahn--Hilliard model \eqref{eq:Ca-Hi} has not been  studied in the context of long time-behavior
or metastable dynamics of solutions; the only papers concerning the Cahn--Hilliard equation with $p$-Laplacian are focused on stationary solutions \cite{Dr-Ma-Ta,Takac}.

\item Finally, we consider the more general double well potential  \eqref{eq:F}, which is only $C^1(\R)$ if $\theta\in(1,2)$, and satisfies $F''(\pm1)=0$, if $\theta>2$.
When considering the competition between a double well potential as in \eqref{eq:F} and the $p$-Laplace operator, 
the case $\theta>p$ is known as \emph{supercritical} or \emph{degenerate} case, see \cite{DCDS} and references therein. 
In contrast, the case $\theta=p>1$ ($\theta\in(1,p)$) is called \emph{critical} (\emph{subcritical}).
\end{enumerate}

In order to briefly describe the derivation of \eqref{eq:CH-model}, we recall the one-dimensional continuity equation for the concentration $u$:
\begin{equation}\label{eq:cont}
		u_t+J_x=0,
\end{equation}
where $J$ is its flux. 
In the case of the standard Cahn--Hilliard equation \eqref{eq:Ca-Hi}, the flux $J$ is related to the chemical potential $\mu$ (see \cite{Gurtin}) according to the law
\begin{equation}\label{eq:flux}
	J=-D\mu_x, \qquad \mbox{ where } \qquad \mu=-\e^2u_{xx}+F'(u), \qquad \mbox{ and } \qquad D>0.
\end{equation}
By substituting \eqref{eq:flux} with $D\equiv1$ in \eqref{eq:cont}, we obtain \eqref{eq:Ca-Hi}.
In this paper, we consider a more general version of the equation \eqref{eq:flux} for the flux, given by
\begin{equation}\label{eq:flux2}
	J=-D(u)(-\e^p(|u_x|^{p-2}u_{x})_x+F'(u))_x.
\end{equation}
Notice that \eqref{eq:flux} can be obtained by \eqref{eq:flux2} by choosing $D\equiv1$ and $p=2$.
By combining the continuity equation \eqref{eq:cont} and the equation for the flux \eqref{eq:flux2}, we obtain \eqref{eq:CH-model}.
In the rest of the paper, we consider equation \eqref{eq:CH-model} with initial datum
\begin{equation}\label{eq:initial}
	u(x,0)=u_0(x), \qquad \qquad x\in[a,b],
\end{equation}
and, similarly to the classical case \eqref{eq:Ca-Hi}, we impose that the flux $J$ vanishes at the boundary points $a,b$.
Since in the case of \eqref{eq:CH-model} the flux is given by \eqref{eq:flux2} and $D$ is strictly positive, we consider the homogeneous  boundary conditions
\begin{equation}\label{eq:Neu-p}
	u_{x}=(-\e^p(|u_x|^{p-2}u_{x})_x+F'(u))_x=0, \quad \mbox{ at } x=a,b, \quad \mbox{ for } \,t\geq0.
\end{equation}
As we already mentioned, the boundary conditions \eqref{eq:Neu-p} guarantee that the total mass of the solution is preserved in time: indeed, by integrating the continuity equation \eqref{eq:cont} and using $J(a,t)=J(b,t)=0$, for any $t\geq0$, we deduce $\int_{a}^{b}u_t\,dx=0$, for any $t\geq0$.
Notice that if $F\in C^2(\R)$ (for instance, if $\theta\geq2$ in \eqref{eq:F}), then the boundary conditions \eqref{eq:Neu-p} are equivalent to 
\begin{equation*}
	u_{x}=(|u_x|^{p-2}u_{x})_{xx}=0, \quad \mbox{ at } x=a,b, \quad \mbox{ for } \,t\geq0;
\end{equation*}
thus, if $F\in C^2(\R)$ and $p=2$,  \eqref{eq:Neu-p} reduces to \eqref{eq:Neu}.

\subsection{Presentation of the main results}\label{sec:intro-main}
The Cahn--Hilliard equation \eqref{eq:Ca-Hi} is closely related to the Allen--Cahn equation, which is another model used to describe
phase transitions and in the one-dimensional case reads as
\begin{equation}\label{eq:Al-Ca}
	u_t=\e^2u_{xx}-F'(u),
\end{equation}
where $\e>0$ is the diffusion coefficient and the potential $F$ is as before. 
In particular, equation \eqref{eq:Al-Ca} can be seen as the gradient flow of the Ginzburg--Landau energy functional \eqref{GL} in $L^2(a,b)$; 
as a consequence, the solutions to \eqref{eq:Al-Ca} do not conserve mass.
The aforementioned metastable dynamics of the solutions to \eqref{eq:Al-Ca}  was firstly investigated   in the celebrated articles \cite{Bron-Kohn,Carr-Pego,Fusco-Hale},
where the authors proposed two different approaches to rigorously studied the slow motion of the solutions.
Subsequently, the same approaches have been applied to many different evolutions PDE, including the Cahn--Hilliard equation \eqref{eq:Ca-Hi}:
being impossible to recall all the contributions, we only mention a very abridged list.
In addition to the already mentioned papers on  metastability for  Cahn--Hilliard models \cite{AlikBateFusc91,Bates-Xun1,Bates-Xun2,Bron-Hilh},
we recall the fundamental work \cite{Grant}, where the author considers the vectorial version of \eqref{eq:Ca-Hi}, known as \emph{Cahn--Morral system}.
More recently, metastable dynamics has been studied for hyperbolic versions of \eqref{eq:Ca-Hi} in \cite{MMAS19,JDDE2021}
and for reaction diffusion equations involving the $p$-Laplace operator in \cite{DCDS}, 
to which we refer the reader in search of a more detailed list of PDEs sharing the phenomenon of metastability.

Inspired by the results contained in \cite{DCDS}, where the reaction-diffusion model 
\begin{equation}\label{eq:Al-Ca-p}
	u_t=\e^p(|u_x|^{p-2}u_x)_x-F'(u),
\end{equation}
with $F$ given by \eqref{eq:F}, is considered and where it is rigorously proved that the evolution of the solutions strongly depend on the interplay between
the parameter $p>1$ and the power $\theta>1$ appearing in the definition \eqref{eq:F} of $F$, we aim to extend such results to the model \eqref{eq:CH-model}.
In particular, the main results of this paper can be sketched as follows. 
To start with, we consider the stationary problem associated to \eqref{eq:CH-model}-\eqref{eq:Neu-p} and, particularly, we focus on two types of steady states:
\begin{itemize}
\item The first ones already appeared in \cite{DCDS} and they exist only in the subcritical case $\theta\in(1,p)$.
We will see that for any natural number $N\geq1$ and any locations $a<h_1<h_2<\dots<h_N<b$, we can choose $\e_0>0$ small enough so that for any $\e\in(0,\e_0)$,
there exist two steady states of \eqref{eq:CH-model}-\eqref{eq:Neu-p} that attain both the values $\pm1$ with exactly $N$ zeroes \emph{arbitrarily} located at $h_1,\dots,h_N$.
\item The second ones are peculiar of the model \eqref{eq:CH-model} with $p>2$,
as they are neither solutions of the generalized Allen--Cahn model \eqref{eq:Al-Ca-p} nor of the standard Cahn--Hilliard model \eqref{eq:Ca-Hi}.
These steady states can have an arbitrary number of zeroes as before, but their location is not arbitrary 
since they consist of a chain of \emph{pulse} solutions suitably glued together.
\end{itemize}

After studying the stationary problem and proving that there exist steady states with an arbitrary number of transitions between $\pm1$
located at random points in $(a,b)$ only in the subcritical case $\theta\in(1,p)$, 
we thus focus the attention on the case $\theta\geq p>1$; 
here, since the steady states have transition points that are not randomly located,
there exist solutions which are neither stationary nor they are close to them, but still evolve very slowly in time.
To be more specific:

\begin{itemize}
\item In the critical case $\theta=p>1$, we extend to the generalized Cahn--Hilliard equation \eqref{eq:CH-model}
the classical results on the exponentially slow motion of the solutions to \eqref{eq:Ca-Hi}.
Precisely, we prove that there exists solutions with $N$ transitions between $\pm1$ which maintain such a layered structure
for times of $\mathcal{O}\left(e^{c/\e}\right)$, with $c>0$ (the so called \emph{metastable states}).
\item In the supercritical case $1<p<\theta$, we prove that layered structures still evolve slowly in time, but only with an algebraically small speed,
that is they maintain their unstable configurations for times of $\mathcal{O}\left(\e^{-k}\right)$, with $k>0$.
It is worth noticing that these results are new also for the classical Cahn--Hilliard equation \eqref{eq:Ca-Hi} ($\theta>p=2$).
\end{itemize}

In order to prove the slow motion results sketched above, 
we mean to adapt a strategy firstly introduced by Bronsard et al. in \cite{Bron-Hilh, Bron-Kohn},
and then improved by Grant in \cite{Grant}, where the author is able to prove exponentially slow motion of solutions to the Cahn--Morral system.
The key point of such a strategy hinges on the use of the \emph{normalized energy functional}
\begin{equation}\label{eq:energy-p=2}
	\int_a^b\left[\frac{\e|u_x|^2}{2}+\frac{F(u)}\e\right]\,dx,
\end{equation}
obtained multiplying by $\e^{-1}$ the Ginzburg--Landau functional \eqref{GL},
this being the reason why the strategy proposed in \cite{Bron-Hilh, Bron-Kohn} is known as \emph{energy approach}.
After its introduction, the quite elementary but powerful energy approach has been applied to study metastable dynamics for many different evolution PDEs:
for an abridged list we refer the reader to the aforementioned articles \cite{MMAS19,DCDS}, and references therein. 
To adapt the energy approach to the model \eqref{eq:CH-model}-\eqref{eq:Neu-p}, we shall use the functional
\begin{equation}\label{eq:energy}
	E_\e[u]:=\frac{1}{\e}\int_a^b\hat\Psi(u,u_x)\,dx=\int_a^b\left[\frac{\e^{p-1}|u_x|^p}{p}+\frac{F(u)}\e\right]\,dx,
\end{equation}
where $\hat\Psi$ is the free-energy introduced in \eqref{eq:free-energy}. 
As we will see in Section \ref{sec:slow}, the energy \eqref{eq:energy} plays the same crucial role played by \eqref{eq:energy-p=2} for \eqref{eq:Ca-Hi}
and it allows us to prove either the exponentially or the algebraic slow motion of solutions to \eqref{eq:CH-model}-\eqref{eq:Neu-p}.

\subsection*{Plan of the paper}
The rest of the paper is structured as follows.
In Section \ref{sec:steady} we study the stationary problem associated to \eqref{eq:CH-model}-\eqref{eq:Neu-p},
with the aim of showing that steady states with an arbitrary number of transitions located at arbitrary positions in $(a,b)$ can exist only in the subcritical case $\theta\in(1,p)$.
Moreover, in Proposition \ref{prop:2trans} we prove existence of \emph{pulse} solutions in the case $p>2$;
these solutions can be suitably glued together to obtain solutions with a generic number of transitions ($N\geq3$),
whose positions must satisfy a certain property, for details see Section \ref{rem:2trans}.  
Section \ref{sec:slow} is devoted to the study of the slow evolution of solutions with a layered structure.
In Theorem \ref{thm:main} we consider the case $\theta=p$, and we prove persistence of metastable states for an exponentially long time;
the algebraic slow motion in the case $\theta>p$ is proved in Theorem \ref{thm:main2}.
Finally, in Section \ref{LD} we provide an estimate on the velocity of the transition points,
showing that they move with either exponentially or algebraically small speed if $\theta=p$ or $\theta>p$, respectively (cfr. Theorem \ref{thm:main3}).

\section{Steady states}\label{sec:steady}
Studying the stationary problem associated to the model \eqref{eq:CH-model}-\eqref{eq:Neu-p} is an interesting and difficult topic,
just think that there is a vast literature devoted to the particular case $p=2$ and a non-degenerate double well potential as in \eqref{eq:F} with $\theta=2$, 
corresponding to the standard Cahn--Hilliard equation \eqref{eq:Ca-Hi}.
An abridged list of references includes \cite{AlikBateFusc91,Bates-Xun2,CarrGurtSlem,Grin-Nov,KosMorYot,Nov-Pel,Zheng}.
The aim of this section is to understand whether a function with an arbitrary number $N\in\mathbb N$ of transitions,
located at arbitrary positions $a<h_1<h_2<\dots<h_N<b$ is a steady state of \eqref{eq:CH-model}-\eqref{eq:Neu-p}.
This study is preliminary to the main results of this paper, which are contained in the following sections, 
when we prove slow motion of solutions with a transition layer structure, that are neither stationary nor they are close to them.

From \eqref{eq:CH-model}, it follows that stationary solutions satisfy
\begin{equation*}
	\left[D(u)\left(-\varepsilon^p(|u_x|^{p-2}u_x)_x+F'(u)\right)_x\right]_x=0, \qquad \qquad x\in(a,b),
\end{equation*}
and, as a consequence of the boundary conditions \eqref{eq:Neu-p}, since $D$ is strictly positive, 
all the stationary solutions to \eqref{eq:CH-model}-\eqref{eq:Neu-p} satisfy the boundary value problem (BVP)
\begin{equation}\label{eq:BVP}
	\begin{cases}
		-\varepsilon^p(|u_x|^{p-2}u_x)_x+F'(u)=\beta\in\R, \qquad \qquad x\in(a,b),\\
		u_x(a)=u_x(b)=0.
	\end{cases}
\end{equation}
Hence, for any fixed $\beta\in\R$, a solution to \eqref{eq:BVP} gives a steady state of \eqref{eq:CH-model}-\eqref{eq:Neu-p};
for instance, notice that any real constant provides a steady states of such a problem.
In the particular case $\beta=0$, we obtain the steady states of the reaction-diffusion model \eqref{eq:Al-Ca-p}, with homogeneous Neumann boundary conditions, 
that has been already studied in previous works, see \cite{Dr-Ma-Ta,Dra-Rob,Dra-Rob2,DCDS} among others. 
For completeness, we briefly recall the results contained in the latter articles.
If $F$ is given by \eqref{eq:F} with $\theta\geq p>1$, the set of all solutions is qualitatively the same as the case $\theta=p=2$,
that is the classic boundary value problem with linear diffusion and a double well potential with wells of equal depth, namely
$$\varepsilon^2u_{xx}+u-u^3=0, \qquad x\in(a,b), \qquad \qquad u_x(a)=u_x(b)=0.$$
It is well known that the only solutions to such boundary value problem are the constant solutions $u=-1,0,1$, and non constant solutions that can be extended to
{\bf periodic functions} on $\R$, which always satisfy $-1<u(x)<1$, for any $x\in(a,b)$ (for further details see \cite{Carr-Pego}). 
Such characterization is preserved also if one considers a $p$-Laplace operator and a potential $F$ as in \eqref{eq:F}, but only in the case $\theta\geq p>1$ (see \cite{DCDS}).
In contrast, if $1<\theta<p$, the structure of the set of stationary solutions is much richer, and there exist steady states that attain both the values $\pm1$ with an arbitrary number of transitions located at {\bf arbitrary positions} in $(a,b)$ (and therefore they are not necessarily periodic).
To be more precise, we recall the following result contained in \cite{DCDS}.

\begin{prop}\label{prop:teta<p}
Let us consider the BVP \eqref{eq:BVP} with $\beta=0$, $F$ given by \eqref{eq:F} and $1<\theta<p$. 
Fix $N\in\mathbb{N}$ and $a<h_1<h_2<\dots<h_N<b$. 
There exists $\e_0>0$ such that if $\e\in(0,\e_0)$, then there exist two solutions $\pm\varphi^\e_N$ 
to \eqref{eq:BVP} satisfying  $|\varphi^\e_N| \leq 1$ and with exactly $N$ zeros in $h_1,\dots,h_N$.

Moreover, for any $\e\in(0,\e_0)$, we have
\begin{equation}\label{eq:energy-teta<p}
	E_\e[\pm\varphi^\e_N]=Nc_p, \qquad \qquad c_p:=\left(\frac{p}{p-1}\right)^{\frac{p-1}{p}}\int_{-1}^{+1} {F(s)}^\frac{p-1}{p}\,ds,
\end{equation}
where the energy $E_\e$ is defined in \eqref{eq:energy}. 
\end{prop}
\begin{proof}
For the proof of the existence see \cite[Proposition 2.5]{DCDS}, while for the proof of \eqref{eq:energy-teta<p} see \cite[Proposition 3.7 and Remark 3.8]{DCDS}.
\end{proof}

An interesting related problem is whether the steady states of Proposition \ref{prop:teta<p} are dynamically stable under small perturbations, 
inasmuch as it has been recently proved that they are unstable as variational solutions to the associated elliptic problem, see Theorem 1.5 in \cite{DPV20}. 
In other words, such critical points are not strict local minimizers because of \eqref{eq:energy-teta<p} and we imagine two possible scenarios:
either a small perturbation force the corresponding time-dependent solution of \eqref{eq:Al-Ca-p} to evolve until it reaches the global minimum of the energy \eqref{eq:energy},
that is equal to zero, or a small perturbation does not destroy the transition layer structure, which is maintained for all times $t\geq0$.
In \cite{Dra-Rob2}, a characterization of a subset of the basin of attraction for the aforementioned local minimizers is provided.

On the other hand, as it was already mentioned, if $\theta\geq p$ solutions as in Proposition \ref{prop:teta<p} can not exist
because the only non constant solutions can be extended to periodic solutions on $\R$.
For an arbitrary number $N\in\mathbb{N}$, there exist solutions taking values in $(-1,1)$ with exactly $N$ transitions,
but layer positions must repeat in a regular fashion, meaning that they can not be arbitrary chosen;
indeed, these solutions can be seen as truncations of periodic solutions on the whole real line of period $2(b-a)/N$.

We now focus on the problem \eqref{eq:BVP} for $\beta\neq0$: in order to understand the structure of its solutions, we study the equation in the whole real line,  that is
\begin{equation}\label{eq:ODE}
	\varepsilon^p(|u_x|^{p-2}u_x)_x-G_\beta'(u)=0, \qquad \mbox{in $\R$},  \qquad \quad \mbox{ where } \quad G_\beta(u):=F(u)-\beta u.
\end{equation}
Notice that $G_0=F$, so $G_\beta$ is a balanced double well potential for $\beta=0$, while $G_\beta$ is an unbalanced double well potential for $\beta\neq0$,
see Figure \ref{fig:G_b}.

\begin{figure}[ht]
\centering
\includegraphics[width=5.5cm,height=4cm]{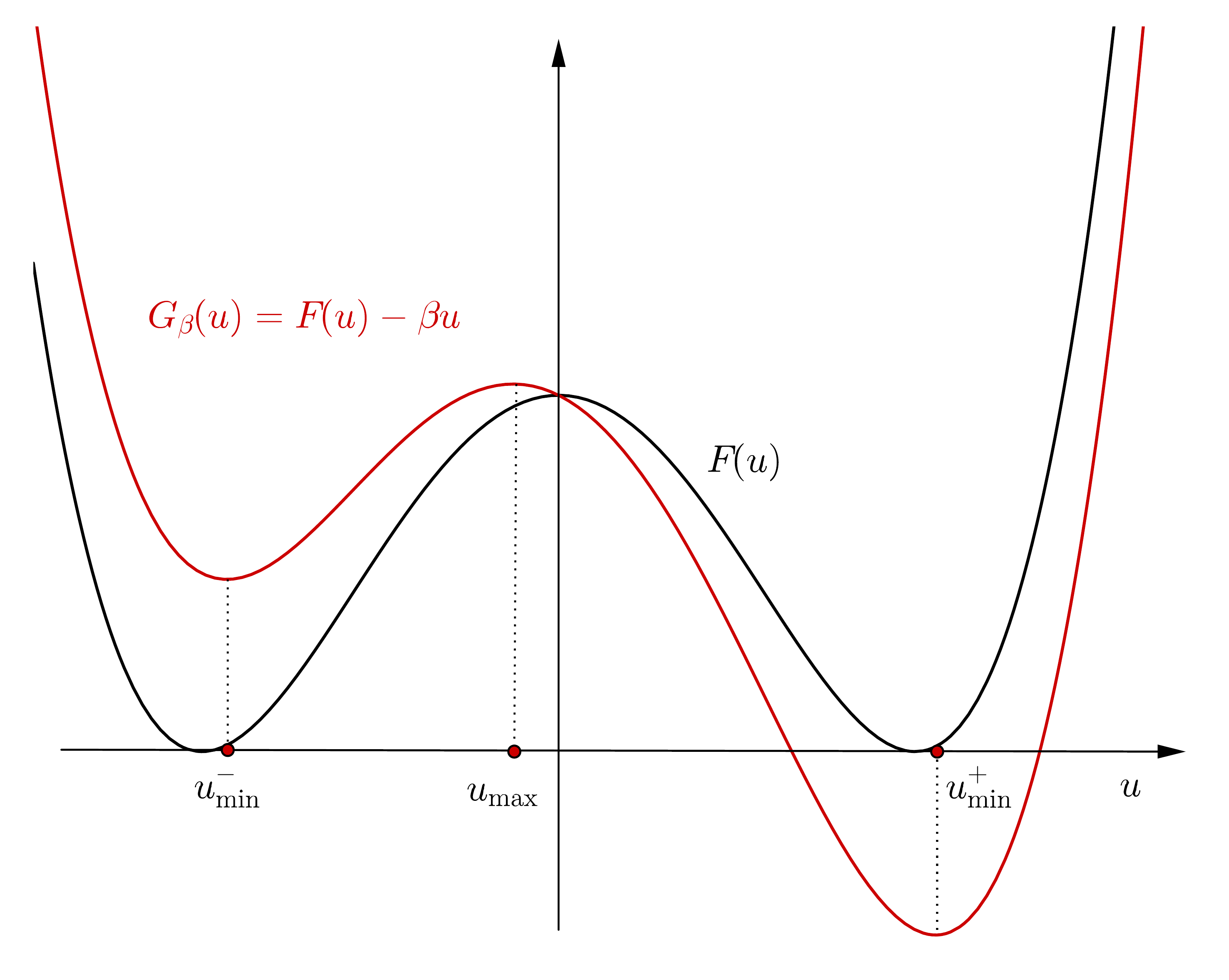}
 \qquad\quad
\includegraphics[width=5.5cm,height=4cm]{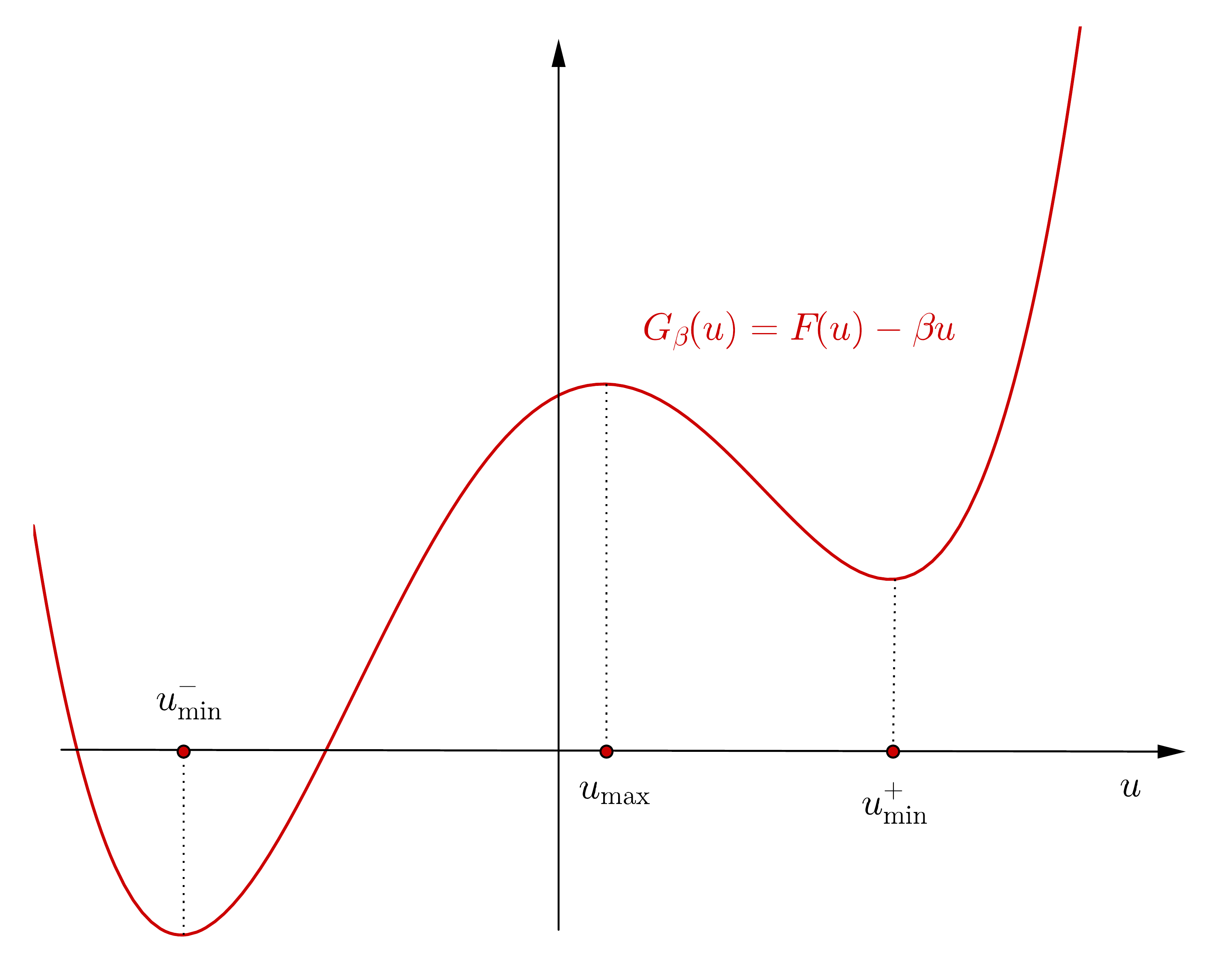}
 \hspace{3mm}
 \caption{\small{In the above pictures $\theta=2$. For comparison, on the left we plot the graph of $F(u)=|1-u^2|^2$ (black line) and $G_\beta(u)= F(u)-\beta u$, with $\beta>0$  (red line). The potential still have two minima (named here $u_{\min}^- <u_{\min}^+$), but $F(u_{\min}^- ) \neq F(u_{\min}^+)$ (unbalanced potential). We can also notice a translation of all the critical points, including the maximum $u_{\max}$. The same happens if $\beta <0$ (right plot).}} 
 \label{fig:G_b}
 \end{figure}
To be more precise, we have
$$G'_\beta(u)=F'(u)-\beta=-u(1-u^2)|1-u^2|^{\theta-2}-\beta,$$
and, for $\theta\geq2$, 
$$G''_\beta(u)=F''(u)=\left[(2\theta-1)u^2-1\right]|1-u^2|^{\theta-2}.$$
As a consequence, $G_\beta$ has exactly three critical points for any $\beta\in(F'(u_+),F'(u_-))$, where 
\begin{equation}\label{eq:u_pm}
	u_\pm:=\pm(2\theta-1)^{-1/2}, \qquad  \mbox{ and } \qquad F'(u_\pm)=\mp(2\theta-1)^{1/2-\theta}(2\theta-2)^{\theta-1}.
\end{equation}
Indeed, it is easy to check that  the equation $G'_\beta(u)=0$ (i.e. $F'(u)=\beta$) has exactly three solutions if and only if $\beta\in(F'(u_+),F'(u_-))$.
Moreover, the equation $G''_\beta(u)=0$ has exactly two solutions $u=u_\pm$, if $\theta=2$, and four solutions $u=\pm1,u_\pm$, for $\theta>2$.

Multiplying by $u_x$ the ODE \eqref{eq:ODE}, we deduce, for $x \in \R$
\begin{equation}\label{eq:first-ode}
	\frac{\e^p(p-1)}{p}|u_x|^p=G_\beta(u)-\kappa, \qquad \qquad \mbox{ with }\, \kappa\in\R.
\end{equation}

\begin{figure}[ht]
\centering
\includegraphics[width=7cm,height=4.5cm]{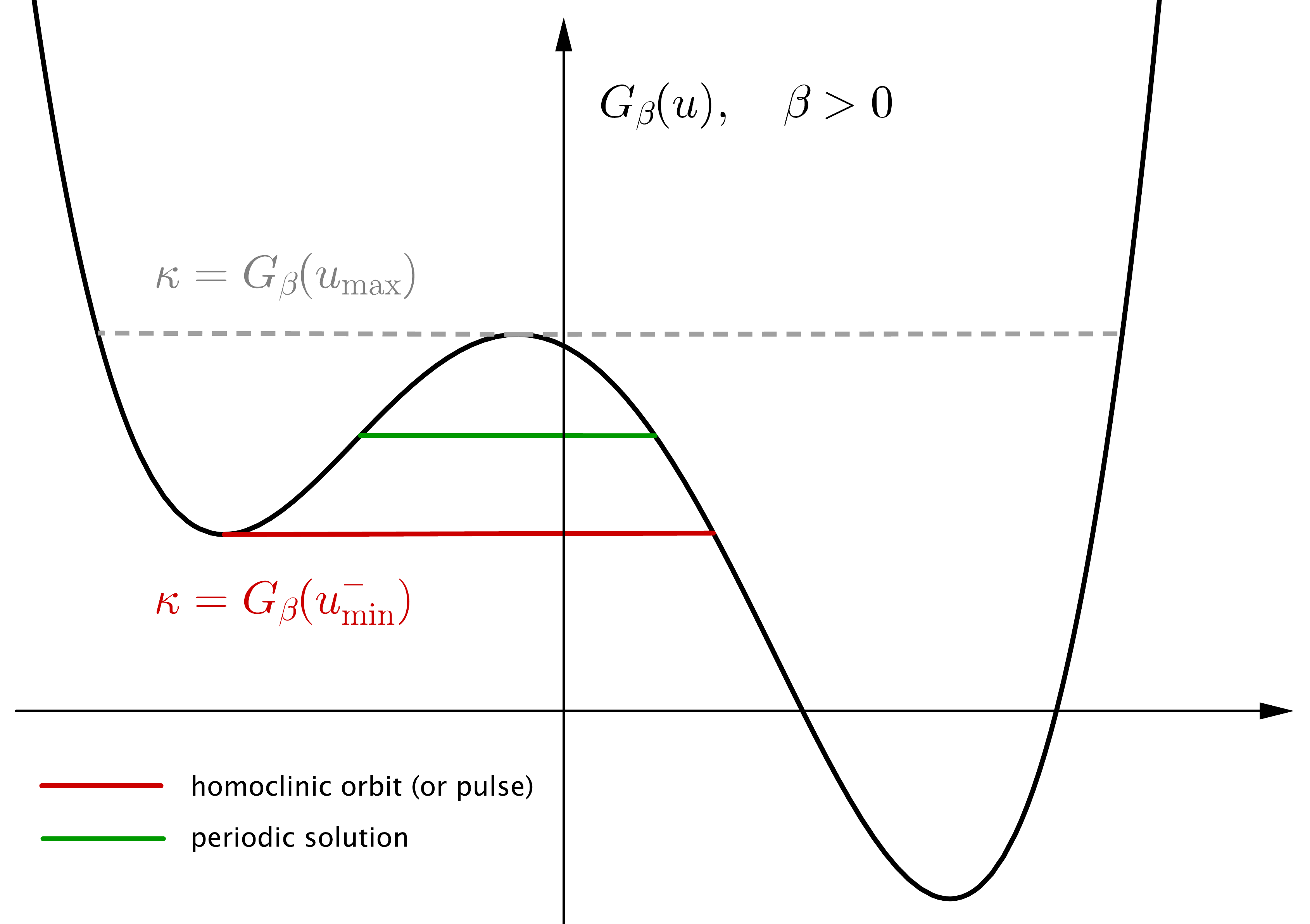}
 \hspace{3mm}
 \caption{\small{Bounded solutions to \eqref{eq:ODE} correspond to different levels of the constant $\kappa\in[G_\beta(u^-_{\min}),G_\beta(u_{\max})]$; 
 in particular, for $\kappa=G_\beta(u_{\min}^-)$ and $\beta\in(0,F'(u_-))$, we have the homoclinic orbit (or \emph{pulse}) joining the critical point $u_{\min}^-$ to itself. 
 Similarly, if $\beta<0$, a homoclinic orbit connecting $u_{\min}^+$ to itself appears. 
 In both cases the orbits touch the values $u_{\min}^{\pm}$ only for $ p>2$.
 }} 
 \label{fig:pp}
 \end{figure}
 
From the phase portrait in Figure \ref{fig:pp}, it is clear that the boundary conditions in \eqref{eq:BVP} imply that all  solutions to this problem lie on closed orbits, this being the reason why we are interested in studying solutions to \eqref{eq:first-ode} for $\kappa \!\in \!\left[G_{\beta}(u_{\min}^-),G_\beta(u_{\max})\right]$.
As a consequence, stationary solutions to \eqref{eq:CH-model}-\eqref{eq:Neu-p} correspond to appropriate choices of the parameters $\beta,\kappa$ in \eqref{eq:first-ode}; for example, when choosing $$\kappa \in \left \{ G_\beta(u_{\min}^-), G_\beta(u_{\max}), G_\beta(u_{\min}^+) \right\}, $$ we have constant steady states
and, as in the case $\beta=0$, there are non constant solutions which can be seen as a truncation of periodic solutions in the whole real line, 
corresponding to $\kappa\in\left(G_{\beta}(u_{\min}^-), G_\beta(u_{\max})\right)$, see the green line in Figure \ref{fig:pp}.
However, taking advantages of the pair $(\beta,\kappa)\neq(0,0)$, we can construct many different solutions.
As it was already mentioned, this is a very ambitious goal we do not accomplish in this paper;
anyway, for the interested reader we refer to the aforementioned articles \cite{AlikBateFusc91,Bates-Xun2,CarrGurtSlem,Grin-Nov,KosMorYot,Nov-Pel,Zheng},
where the case $p=\theta=2$ is considered in detail.
Here, we only recall that if we add a mass constraint to \eqref{eq:BVP} of the form
\begin{equation}\label{eq:mass-constraint}
 	\frac{1}{b-a}\int_a^b u(x)\,dx=m,
\end{equation}
we can assert that for any fixed $N\in\mathbb N$ and $m\in(-1,1)$, it is possible to choose $\e>0$ sufficiently small such that
there are solutions with $N$ transitions satisfying \eqref{eq:mass-constraint}.
It has to be observed again that if $N\geq2$, the location of the transition layers is {\bf not arbitrary}:
if $h^e:=(h^e_1,\dots,h^e_N)$ is the vector of layer locations, that is the stationary solution satisfies $u^e(h^e_i)=0$, 
for $i=1,\dots,N$ and $a<h^e_1<\dots<h^e_N<b$, then $u^e$ is the periodic extension of that part of $u^e$ in $[a,(h^e_2+h^e_3)/2]$
and one has $l^e_i=l^e_{i+2}$, for any $i=1,\dots,N-1$, where $l^e_i=h^e_i-h^e_{i-1}$, $h^e_0:=2a-h^e_1$ and $h^e_{N+1}:=2b-h^e_N$.
On the other hand, if $N=1$ the position of the single transition is arbitrary and monotone solutions play a crucial role, 
because profiles with more than one transition can not minimize the energy, for details see \cite{CarrGurtSlem}.

Coming back to our problem \eqref{eq:BVP} for generic $\theta$ and $p$, in the case $1<\theta<p$  Proposition \ref{prop:teta<p} ensures the existence of solutions oscillating between $\pm 1$ and with $N \in \N$ transitions that are arbitrarily located; 
this is a consequence of the fact that the heteroclinic orbit, peculiar of a balanced potential and corresponding to the choice $(\beta,\kappa)=(0,0)$, attains both the values $\pm1$.

We thus focus on the case $1<p\leq\theta$, where the heteroclinic orbit does not attain $\pm1$, but it converges asymptotically to them as $x\to\pm\infty$; on the one hand,  we expect a similar result as the case $\theta=p=2$ to hold true.
On the other hand, if $p>2$, we can construct \emph{new} solutions to {\eqref{eq:BVP}} by {\it truncating} the homoclinic {orbits} of \eqref{eq:ODE}, 
which are peculiar of this problem, since we are dealing with an unbalanced potential, see Figure \ref{fig:2trans}.

\begin{figure}[ht]
\centering
\includegraphics[width=6cm,height=4cm]{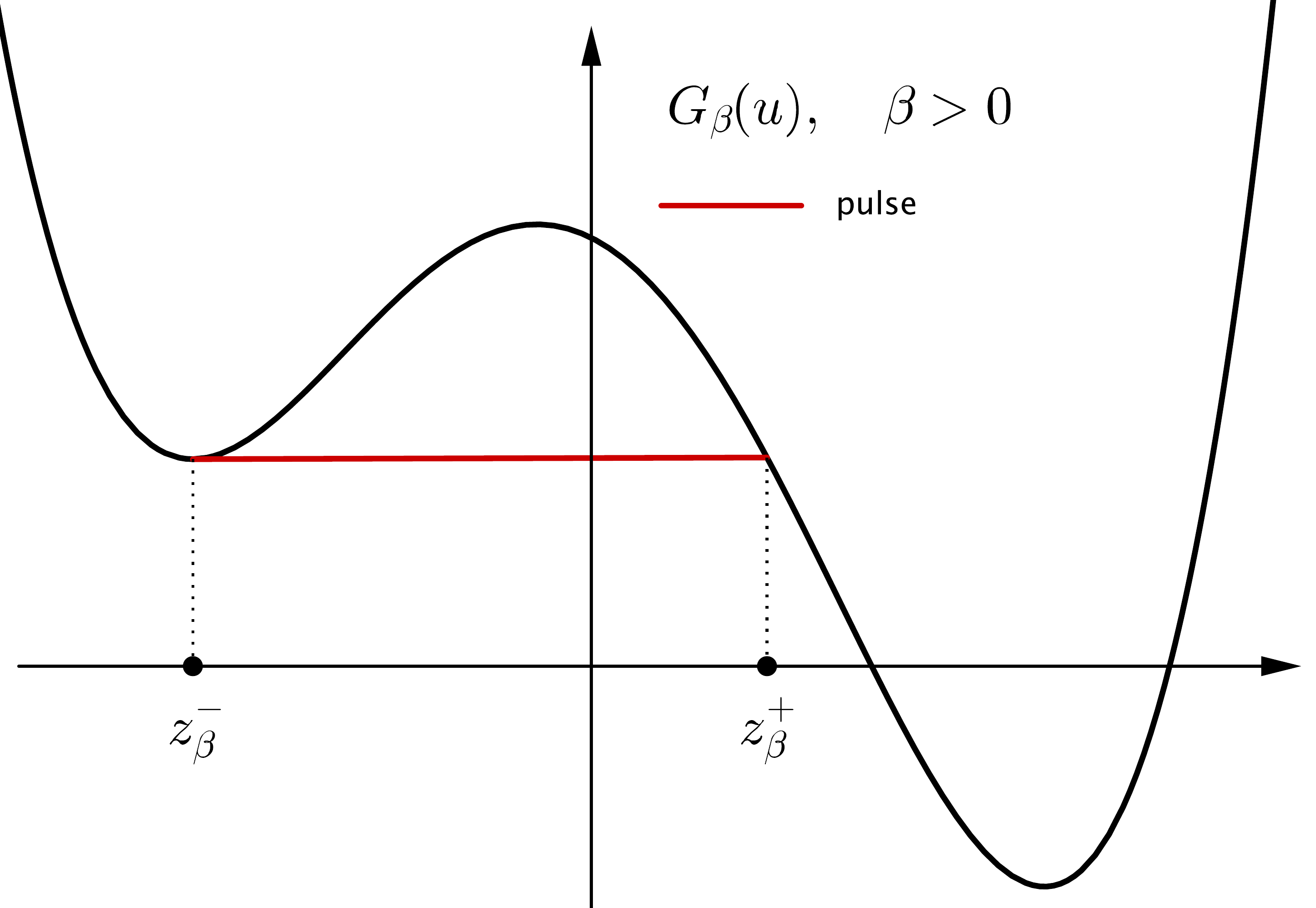}
 \hspace{3mm}
 \caption{\small{The red line corresponds to the pulse starting from $z_\beta^-$ and with a first transition point, where the solution is zero;
 once the solution becomes positive, it reaches its maximum value $z_\beta^+$ and, after that, comes back to $z_\beta^-$, experiencing a second transition.
 }} 
 \label{fig:2trans}
 \end{figure}

In the  following result we provide the existence of  stationary solutions to  \eqref{eq:ODE} with exactly two transitions; 
we stress once again that these solutions, appropriately truncated, satisfy the BVP \eqref{eq:BVP} as well.

\begin{prop}[existence of pulse solutions in the case $p>2$]\label{prop:2trans}
Let $u_\pm\in\R$ given by \eqref{eq:u_pm} and for any  {$\beta\in(0,F'(u_-))$}, set 
$$z_\beta^-:=\min_\R\left\{z\in\R : G'_\beta(z)=0\right\}.$$
If $2 < p \leq \theta$, then there exists a pulse $\psi_\beta:=\psi_\beta(x)$ satisfying \eqref{eq:ODE},
$$\psi_\beta(0)=\max_\R \psi_\beta=: z_\beta^+,$$
with $G_\beta(z_\beta^+)=G_\beta(z_\beta^-)$, and
\begin{align*}
	\psi_\beta(x)=z_\beta^-, \qquad \forall\,x\in(-\infty,-\omega_\beta]\cup[\omega_\beta,+\infty), \\
	\psi'_\beta>0, \quad \mbox{ in } \, (-\omega_\beta,0), \qquad \psi'_\beta<0, \quad \mbox{ in } \, (0,\omega_\beta),
\end{align*}
for some $\omega_\beta>0$.
Moreover, we have the following limits
\begin{equation}\label{eq:z_beta}
	\lim_{\beta\to0^+}z_\beta^\pm=\pm1.
\end{equation}
Similarly, if $\beta\in(F'(u_+),0)$, denote by 
$$z_\beta^+:=\max_\R\left\{z\in\R : G'_\beta(z)=0\right\}.$$
Then there exists a pulse $\psi_\beta:=\psi_\beta(x)$ satisfying \eqref{eq:ODE}, 
$$\psi_\beta(0)=\min_\R \psi_\beta=:z_\beta^-,$$
with $G_\beta(z_\beta^-)=G_\beta(z_\beta^+)$, and
\begin{align*}
	\psi_\beta(x)=z_\beta^+, \qquad \forall\,x\in(-\infty,-\omega_\beta]\cup[\omega_\beta,+\infty), \\
	\psi'_\beta<0, \quad \mbox{ in } \, (-\omega_\beta,0), \qquad \psi'_\beta>0, \quad \mbox{ in } \, (0,\omega_\beta),
\end{align*}
for some $\omega_\beta>0$.
\end{prop}

\begin{proof}
We prove the result in the case $\beta\in(0,F'(u_-))$, being the other case very similar.
Choose $\kappa=G_\beta(z^-_\beta)$ in \eqref{eq:first-ode};
notice for later use that, by definition, $G_\beta'(z^-_\beta)=0$ and, since $z^-_\beta$ is an increasing function of $\beta$ satisfying
$$\lim_{\beta\to0^+}z^-_\beta=-1, \qquad \mbox{ and } \qquad \lim_{\beta\to F'(u_-)}z^-_\beta=u_-=-(2\theta-1)^{-1/2}<0,$$
we have 
\begin{equation}\label{eq:G''-}
	G_\beta''(z^-_\beta)>0, \qquad \mbox{ for any } \beta\in(0,F'(u_-)).
\end{equation} 
Moreover, let use denote by $z^+_\beta$ the only point such that 
\begin{equation}\label{eq:G'+}
	G_\beta(z^+_\beta)=G_\beta(z^-_\beta), \qquad \quad G'_\beta(z^+_\beta)<0, \qquad \mbox{ for any } \beta\in(0,F'(u_-)),
\end{equation} 
see Figure \ref{fig:2trans} above.
From \eqref{eq:first-ode}, it follows that
$$\e^p|u_x|^p=\frac{p}{p-1}\left[G_\beta(u)-G_\beta(z^-_\beta)\right],$$
and the function $\psi_\beta$ we are looking for is implicitly defined by
$$\int_{\psi_\beta(x)}^{z^+_\beta}\left[G_\beta(s)-G_\beta(z^-_\beta)\right]^{-1/p}\,ds=\left(\frac{p}{p-1}\right)^{1/p}\e^{-1}|x|.$$
We claim that if $p>2$, then
\begin{equation}\label{eq:claim-int}
	\int_{z^-_\beta}^{z^+_\beta}\left[G_\beta(s)-G_\beta(z^-_\beta)\right]^{-1/p}\,ds<\infty,
\end{equation}
and the thesis holds true with
\begin{equation*}
	\omega_\beta:=\e\left(\frac{p-1}{p}\right)^{1/p}\int_{z^-_\beta}^{z^+_\beta}\left[G_\beta(s)-G_\beta(z^-_\beta)\right]^{-1/p}\,ds.
\end{equation*}
In order to prove \eqref{eq:claim-int}, let us consider the two integrals
$$I^+:=\int_{0}^{z^+_\beta}\left[G_\beta(s)-G_\beta(z^-_\beta)\right]^{-1/p}\,ds,  \qquad  
I^-=\int_{z^-_\beta}^{0}\left[G_\beta(s)-G_\beta(z^-_\beta)\right]^{-1/p}\,ds.$$
Concerning $I^+$, we use \eqref{eq:G'+} and the expansion 
$$G_\beta(s)=G_\beta(z^-_\beta)+G'_\beta(z^+_\beta)(s-z^+_\beta)+\mathcal{O}\left(|s-z^+_\beta|^2\right),$$
to deduce that 
$$I^+:=\int_{0}^{z^+_\beta}\left[G_\beta(s)-G_\beta(z^-_\beta)\right]^{-1/p}\,ds\sim\int_{0}^{z^+_\beta}\left(z^+_\beta-s\right)^{-1/p}\,ds<\infty,$$
for any $p>1$.
On the hand, using \eqref{eq:G''-} and the expansion
$$G_\beta(s)-G_\beta(z^-_\beta)=\frac{G''_\beta(z^+_\beta)}2(s-z^+_\beta)^2+\mathcal{O}\left(|s-z^+_\beta|^3\right),$$
we infer
$$I^-=\int_{z^-_\beta}^{0}\left[G_\beta(s)-G_\beta(z^-_\beta)\right]^{-1/p}\,ds\sim\int_{z^-_\beta}^0\left(s-z^-_\beta\right)^{-2/p}\,ds<\infty,$$
if and only if $p>2$.
Hence, \eqref{eq:claim-int} holds true if and only if $p>2$ and the proof is complete.
\end{proof}

The crucial point of Proposition \ref{prop:2trans} is that the pulse attains the value $z^-_\beta$ ($z^+_\beta$) in the case $\beta>0$ ($\beta<0$):
for definiteness, if considering $\beta>0$, one has 
\begin{equation}\label{eq:attained}
	\psi_\beta(\pm\omega_\beta)=z_\beta^-, \qquad  \psi'_\beta(\pm\omega_\beta)=0, \qquad\qquad \mbox{ for some }\, \omega_\beta>0.
\end{equation}
Thanks to \eqref{eq:attained}, we can construct solutions to  \eqref{eq:BVP} on any bounded interval $(a,b)$, see the subsequent Section \ref{rem:2trans}; 
this is a consequence of the presence of the $p$-Laplacian with $p>2$. 
Indeed, if $p \in (1,2]$, there are not any homoclinic orbits satisfying \eqref{eq:attained}: in this case, the pulse satisfies \eqref{eq:attained} with $\omega_\beta=+\infty$ and it can not be \emph{truncated} to obtain a solution to \eqref{eq:BVP} in any bounded interval $(a,b)$, because the derivative vanishes only in one point.

\subsection{Construction of stationary solutions with $N\geq 2$ transition points}\label{rem:2trans} 

We briefly explain how to construct  solutions to \eqref{eq:BVP} for any interval $[a,b]$ starting from the ones presented in Proposition \ref{prop:2trans}: the main idea is to take advantage of the fact that, for any $\beta \in (F'(u_+), F'(u_-))$, we have a homoclinic solution to \eqref{eq:ODE}. For simplicity, we consider the case $\beta \in (0, F'(u_-))$ (corresponding to the homoclinic joining $z_\beta^-$ to itself), being the other case completely symmetric.
\subsection*{The case $N=2$}
In order to construct solutions with two layers, let us start by choosing $\beta$ such that, for any  $a < h_1 < h_2<b$,  the corresponding pulse has two transitions exactly located in $h_1$ and $h_2$; such arbitrary choice of both the transition points is possible because of the behavior, with respect to $\beta$, of the following function 
$$
	d^\e(\beta):=2 \e\left(\frac{p-1}{p}\right)^{1/p}\int_{0}^{z^+_\beta}\left[G_\beta(s)-G_\beta(z^-_\beta)\right]^{-1/p}\,ds.
$$
Indeed, this function represents the ``space"  needed for the solution to go from zero to $z^+_\beta$ and viceversa, and so the distance between the two transitions. 
To be more precise, $d^\e(\beta)$ is a monotone  decreasing function that  enjoys the following properties:
\begin{equation}\label{eq:d^eps}
	\lim_{\beta \to 0^+} d^\e(\beta) = +\infty \qquad \mbox{and} \qquad \lim_{\beta \to F'(u_-)^-} d^\e(\beta) = 0.
\end{equation}
As a consequence, the function $d^\e(\beta)$ attains all the values  in $(0,+\infty)$, meaning that, for all $\e>0$, 
there exists a unique $\bar\beta_\e$ such that  $d^\e(\bar\beta_\e) \equiv h_2-h_1$.
With such a choice of $\bar\beta_\e$, we have thus constructed a solution to \eqref{eq:ODE} with exactly two transitions that are arbitrarily located in $(a,b)$:  however, in order to be sure that such solution also satisfies the homogeneous Neumann boundary conditions in \eqref{eq:BVP},  
we have to require that it attains the value $z_\beta^-$ ``before" the boundary points $x=a$ and $x=b$,
implying that the values $h_1$ and $h_2$ can not be chosen too close to them. In conclusion,  solutions to \eqref{eq:BVP} with $N=2$ layers exist, but the location of transitions {\it can not be completely random}. 

It is very important to notice that the function $d^\e(\beta)$ can be written as $d^\e(\beta):=\e\tilde d(\beta)$, where $\tilde d$ satisfies \eqref{eq:d^eps};
therefore, if $h_1$ and $h_2$ are arbitrary chosen in $(a,b)$ (but not so close to the boundary points) and $\e>0$ is very small, 
then we also need to choose a very small $\bar\beta_\e>0$ and, as a consequence of \eqref{eq:z_beta}, 
we have a pulse with minimum (maximum) value very close to $-1$ ($+1$): in formulas, we have
\begin{equation*}
	\lim_{\e\to0^+}\bar\beta_\e=0, \qquad \mbox{ and } \qquad \lim_{\e\to0^+}z^\pm_{\bar\beta_\e}=\pm1.
\end{equation*}
\subsection*{The case of $N >2$ transition points}
To conclude this section, we briefly mention that  one can construct  solutions to \eqref{eq:ODE} having an  arbitrary number $N >2$ of transition points: heuristically, the idea is to proceed as in the case $N=2$, and to ``glue" together different translations of the pulse previously constructed in Proposition \ref{prop:2trans}.    

We thus start with a pulse connecting the value $z_\beta^-$ to itself and which has two transitions with distance $h_2-h_1$; 
at this point one can ``glue another pulse" and, since $z_\beta^-$ is a critical point for the potential $G_\beta$, the layered solution can remain constantly equal to $z_\beta^-$ in an interval of random length; after that, once the transition occurs,  the solution touches the value $z_\beta^+$, that is not an equilibrium for the equation. Hence, the following transition (from $z_\beta^+$ to $z_\beta^-$) is fixed by the distance $h_2-h_1$, that has to be repeated. To be more precise, the solution has $N$ transitions satisfying:
$$ h_{2i}- h_{2i-1} =h_2-h_1, \qquad   i = 2, \dots ,\left[ \frac{N}{2} \right].$$ 

Hence, also in this case, these solutions have transition points which are not arbitrarily located.

\begin{rem}
It has to be noticed that in the proof of Proposition \ref{prop:2trans} we never used the fact that $\theta \geq p$; hence, all the previous results hold true also in the case $1<\theta<p$, thus providing the existence of  stationary solutions with a generic number $N \in \N$ of  transition points that are not arbitrarily located (as opposite to the ones given by Proposition \ref{prop:teta<p}). Again, we stress that the crucial point is the use of the homoclinic orbits instead of the heteroclinic ones. 
\end{rem}

\section{Slow motion}\label{sec:slow}
The aim of this section is to investigate the slow motion of the solutions to the initial boundary value problem \eqref{eq:CH-model}-\eqref{eq:initial}-\eqref{eq:Neu-p},
when the potential $F$ is given by \eqref{eq:F} with $\theta\geq p>1$.
As we sketched in the Introduction, we rigorously prove the existence of metastable states for the model \eqref{eq:CH-model}-\eqref{eq:Neu-p}, that is the persistence of unstable structures for a very long time $T_\e>0$, satisfying $T_\e\to+\infty$, as $\e\to0^+$, and we show that the slow evolution of the solutions strongly depends on the interplay between the two parameters $\theta,p>1$.
In particular, in the critical case $\theta=p$, we have $T_\e\geq\exp(C/\e)$, with $C>0$ independent of $\e$ (exponentially slow motion) 
and we extend to \eqref{eq:CH-model} the classical results valid for the standard Cahn--Hilliard equation \eqref{eq:Ca-Hi};
on the other hand, in the degenerate case $\theta>p$, the unstable structures persist for a time $T_\e\geq C\e^{-k}$, for some $C,k>0$, 
independent on $\e$, and we only have algebraic slow motion.

Before stating our main results, we present some crucial properties of the energy functional \eqref{eq:energy}, 
that allow us to obtain slow motion of solutions by adapting to our case the energy approach previously mentioned in Section \ref{sec:intro-main}, 
see \cite{Bron-Hilh,Bron-Kohn,MMAS19,DCDS,Grant}.

\subsection{Energy estimates}
From now on, we will use the notation $\tilde u : [a,b]\to\R$ to denote the antiderivative of a generic function $u:[a,b]\to\R$, satisfying $\tilde u(a)=0$.
Hence, if $u$ is the solution to the initial boundary value problem \eqref{eq:CH-model} with initial datum \eqref{eq:initial} 
and boundary conditions \eqref{eq:Neu-p},  we introduce the function
\begin{equation*}\label{eq:u-tilde}
	\tilde{u}(x,t):=\int_a^xu(y,t)\,dy.
\end{equation*}
Clearly, $\tilde{u}(a,t)=0$, for any $t>0$; moreover, since the solution to \eqref{eq:CH-model}-\eqref{eq:initial}-\eqref{eq:Neu-p} preserves the mass and $u_x=\tilde{u}_{xx}$,
we have the following Dirichlet boundary conditions for $\tilde u$:
\begin{equation}\label{eq:bound-ut}
	\tilde u(a,t)=0, \qquad \tilde u(b,t)=\int_a^bu_0(y)\,dy, \qquad \tilde{u}_{xx}(a,t)=\tilde{u}_{xx}(b,t)=0, \qquad \forall\,t\geq0.
\end{equation}
By integrating \eqref{eq:CH-model} and using the boundary conditions \eqref{eq:Neu-p} at $x=a$, we deduce that
\begin{equation}\label{eq:integrated}
	\tilde u_t= D(\tilde u_x)\left(-\e^p(|\tilde u_{xx}|^{p-2}\tilde u_{xx})_x+F'(\tilde u_x)\right)_x,
\end{equation}
where we used the equality $\tilde u_x=u$.


The next result ensures that the energy functional defined in \eqref{eq:energy} is a non-increasing function of time 
if evaluated along a smooth solution to \eqref{eq:CH-model}-\eqref{eq:Neu-p}.
\begin{lem}\label{lem:energy-dec}
Let $u\in C([0,T],H^3(a,b))$ be a solution to \eqref{eq:CH-model}, with $D\in C^1(\R)$ strictly positive and satisfying the boundary conditions \eqref{eq:Neu-p}.
Then,
\begin{equation}\label{eq:der-energy}
	\frac{d}{dt}E_\e[u](t)=-\e^{-1}\int_a^b\frac{\tilde u^2_t(x,t)}{D(u(x,t))}\,dx,
\end{equation}
for any $t\in[0,T]$.
As a consequence, there exists $d>0$ such that
\begin{equation}\label{eq:energy-est}
	E_\e[u](0)-E_\e[u](t)\geq d\e^{-1}\int_0^t\|\tilde u_t(\cdot,s)\|^2_{{}_{L^2}}\,ds,
\end{equation}
for any $t\in[0,T]$.
\end{lem}
\begin{proof}
Since the solution is regular enough, we can differentiate as follows
\begin{equation*}
	\frac{d}{dt}E_\e[u]=\int_a^b\left[\e^{p-1}|u_x|^{p-2}u_x u_{xt}+\frac{F'(u)u_t}{\e}\right]dx.
\end{equation*}
Integrating by parts and using the boundary conditions \eqref{eq:Neu-p}, we obtain
\begin{align*}
	\frac{d}{dt}E_\e[u]&=\int_a^bu_t\left[-\e^{p-1}(|u_x|^{p-2}u_x)_x+\frac{F'(u)}{\e}\right]dx\\
	&=\e^{-1}\int_a^b\tilde u_{xt}\left[-\e^p(|\tilde u_{xx}|^{p-2}\tilde u_{xx})_x+F'(\tilde u_x)\right]dx,
\end{align*}
where we used again the equality $\tilde u_x=u$.
Integrating again by parts, we infer
\begin{equation*}
	\frac{d}{dt}E_\e[u]=\e^{-1}\int_a^b\tilde u_{t}\left[-\e^p(|\tilde u_{xx}|^{p-2}\tilde u_{xx})_x+F'(\tilde u_x)\right]_xdx,
\end{equation*}
where the boundary terms coming from the integration by parts vanish because of \eqref{eq:bound-ut} ($\tilde u(a,t)$ and $\tilde u(b,t)$ do not depend on $t$,
and so, $\tilde u_t(a,t)=\tilde u_t(b,t)=0$, for any $t\in[0,T]$).
Therefore, since $D$ is strictly positive, \eqref{eq:integrated} gives equality \eqref{eq:der-energy};
integrating in $[0,t]$, we end up with \eqref{eq:energy-est}, with
$$d:=\left(\max_{(x,t)\in[a,b]\times[0,T]} D(u(x,t))\right)^{-1},$$
and the proof is complete.
\end{proof}

Thanks to \eqref{eq:energy-est}, we shall prove that it is possible to choose a very large $T_\e>0$ such that
\begin{equation}\label{eq:key}
	\int_0^{T_\e}\|\tilde u_t(\cdot,s)\|^2_{{}_{L^2}}\,ds\leq\sigma(\e),
\end{equation}
with $\sigma:\R^+\to\R^+$ very small.
Hence, the idea is to take advantage of the smallness of the $L^2$-norm of $\tilde u_t(\cdot,t)$ in $[0,T_\e]$ to prove the aforementioned slow motion results.

The strategy to prove \eqref{eq:key} is based on \eqref{eq:energy-est}: 
first, we prove a \emph{lower bound} on the energy \eqref{eq:energy}, then we consider properly assumptions on the initial datum, 
such that the variation of the energy is very small for any $t\in[0,T_\e]$, with $T_\e\to+\infty$, as $\e\to0^+$. 

Following \cite{DCDS}, we make use of the generalized Young inequality
\begin{equation}\label{young}
	a b \leq \frac{a^p}{p}+ \frac{b^q}{q}, \qquad \mbox{with} \quad \frac{1}{p}+\frac{1}{q}=1,
\end{equation}
to deduce
\begin{equation}\label{eq:c_p}
	E_\e[{u}]\geq\int_a^b  |u_x| \left( \frac{p}{p-1} F(u)\right)^{\frac{p-1}{p}} \,dx=\left(\frac{p}{p-1}\right)^{\frac{p-1}{p}}\int_{-1}^{+1} {F(s)}^\frac{p-1}{p}\,ds=:c_p,
\end{equation}
for any $u\in C^1[a,b]$ connecting $-1$ and $+1$.
It is to be observed that when $p=2$, one has 
\begin{equation*}
	c_2=\int_{-1}^{+1}\sqrt{2F(s)}\,ds,
\end{equation*}
which is the minimum energy in the case of the classical Allen--Cahn and Cahn--Hilliard equations \cite{Bron-Hilh, Bron-Kohn}.
The positive constant $c_p$ represents the minimum energy to have a transition between $-1$ and $+1$ in the following sense: 
if a function $u$ is sufficiently close to a function $v$ in some sense to be specified later,  
where $v:[a,b]\to\{-1,1\}$ is a piecewise constant function assuming only the values $\pm1$ with exactly $N$ jumps, then the energy of $u$ satisfies the \emph{lower bound},
\begin{equation}\label{eq:lower-sigma}
	E_\varepsilon[u]\geq Nc_p-\sigma(\varepsilon),
\end{equation}
where $\sigma:\R^+\to\R^+$ is a \emph{small reminder} that depends on $\theta,p$.
More precisely, we will present two different lower bounds of the form \eqref{eq:lower-sigma}, depending on whether $\theta=p$ or $\theta>p$:
\begin{itemize}
	\item if $\theta=p$, then $\sigma(\e)$ is exponentially small, that is $\sigma(\e)=C\exp(-Ap/2\varepsilon)$, for some $A,C>0$ independent on $\e$, 
	for details see Proposition \ref{prop:lower} below.
	\item If $\theta>p$, then $\sigma(\e)$ is algebraically small, that is $\sigma(\e)=C\varepsilon^k$, for some $C,k>0$ independent on $\e$, 
	for details see \ref{prop:lower_deg} below.
\end{itemize}

We stress that \eqref{eq:lower-sigma} is a variational result that depends only on the structure of the energy functional \eqref{eq:energy} 
and in its proof equation \eqref{eq:CH-model} does not play a role.
In fact, this result has been already proved in \cite{DCDS}, but we need a different assumption on the function $u$ in the case $\theta=p$, so that
 we have to slightly modify the proof in such a case.

Let us fix here, and throughout the rest of the paper, $N\in\mathbb{N}$ and a {\it piecewise constant function} $v$ with $N$ jumps as follows:
\begin{equation}\label{vstruct}
	v:[a,b]\rightarrow\{-1,1\}\  \hbox{with $N$ jumps located at } a<h_1<h_2<\cdots<h_N<b.
\end{equation}	
Moreover, we fix $r>0$ such that
\begin{equation}\label{eq:r}
	r<\frac{h_{i+1}-h_i}2, \ \hbox{ for}\ i=1,\dots,N, \qquad   a\leq h_1-r,\qquad h_N+r\leq b.
\end{equation}
Finally, for any $p>1$, define 
\begin{equation}\label{eq:lambda}
	\lambda_p:= 2^{1-\frac{1}{p}}{(p-1)^{-\frac{1}{p}}}.
\end{equation}
We have now all the tools to present the lower bound \eqref{eq:lower-sigma} with a an exponentially small reminder $\sigma$ in the case $\theta=p$.
\begin{prop}\label{prop:lower}
Let $E_\e$ be as in \eqref{eq:energy}, with $F$ given by \eqref{eq:F} and $p=\theta>1$.	
Moreover, fix $v$ as in \eqref{vstruct} and fix $A\in(0,r\sqrt2\lambda_p)$, 
where $r$ satisfies \eqref{eq:r} and $\lambda_p$ is defined in \eqref{eq:lambda}.
Then there exist $\e_0,C,\delta>0$ (depending only on $p,v$ and $A$) such that if $u\in H^1(a,b)$ satisfies
\begin{equation}\label{eq:u-v}
	\|\tilde{u}-\tilde{v}\|_{{}_{L^1}}\leq\delta,
\end{equation}
then for any $\e\in(0,\e_0)$,
\begin{equation}\label{eq:lower}
	E_\varepsilon[u]\geq Nc_p-C\exp(-Ap/2\varepsilon),
\end{equation}
where $c_p$ is defined in \eqref{eq:c_p}.
\end{prop}
\begin{proof}
The only difference with respect to the proof in \cite[Propositions 3.2]{DCDS} is to show that we can choose $\delta$ in the assumption \eqref{eq:u-v}
 such that that $u$ is arbitrary close to $+1$ (or $-1$) in a point.
Hence, we report here only this modification and refer to \cite[Propositions 3.2]{DCDS} for all the details of the proof.  
	
Fix $u\in H^1(a,b)$ satisfying \eqref{eq:u-v}, and take $\hat r\in(0,r)$ and $\rho$ arbitrary small.
Let us focus our attention on $h_i$, one of the points of discontinuity of $v$. To fix ideas, 
let $v(h_i\pm r)=\pm1$, the other case being analogous.
We claim that we can choose $\delta>0$ sufficiently small that there exist $r_+$ and $r_-$ in $(0,\hat r)$ such that
\begin{equation}\label{2points}
	|u(h_i+r_+)-1|<\rho, \qquad \quad \mbox{ and } \qquad \quad |u(h_i-r_-)+1|<\rho.
\end{equation}
Indeed, we have 
\begin{equation*}
	\int_a^b(u-v)w\,dx=-\int_a^b(\tilde{u}-\tilde{v})w'\,dx,
\end{equation*}
for any test function $w\in C^1_c([a,b])$.
Thus, by using \eqref{eq:u-v}, we have 
\begin{equation}\label{eq:keyidea}
	\left|\int_a^b(u-v)w\,dx\right|\leq\|w'\|_{{}_{L^\infty}}\int_a^b|\tilde{u}-\tilde{v}|\,dx\leq\delta\|w'\|_{{}_{L^\infty}},
\end{equation}
for any test function $w\in C^1_c([a,b])$.
Assume by contradiction that $|u-1|\geq\rho$ throughout $(h_i,h_i+\hat{r})$. 
Since $u-v$ is continuous in $(h_i,h_i+\hat{r})$, one has either $u-1\geq\rho>0$ or $u-1\leq-\rho<0$ in the whole interval under consideration. 
Therefore, choosing $w$ non constant and non negative with compact support contained in $(h_i,h_i+\hat{r})$, we obtain
$$\rho\|w\|_{{}_{L^1}}\leq\int_{h_i}^{h_i+\hat{r}}|u-1|w\,dx=\left|\int_a^b(u-v)w\,dx\right|\leq\delta\|w'\|_{{}_{L^\infty}},$$
and this leads to a contradiction if we choose $\delta\in(0,\rho\|w\|_{{}_{L^1}}/\|w'\|_{{}_{L^\infty}})$.
Similarly, one can prove the existence of $r_-\in(0,\hat r)$ such that $|u(h_i-r_-)+1|<\rho_2$.
\end{proof}

Next, we recall the lower bound \eqref{eq:lower-sigma} in the case $\theta>p>1$, cfr. \cite[Proposition 4.1]{DCDS}.

\begin{prop}\label{prop:lower_deg}
Let $p>1$, $F$ given by \eqref{eq:F} with $\theta>p$,
$v:(a,b)\rightarrow\{-1,+1\}$ a piecewise constant function with exactly $N$ discontinuities (as in \eqref{vstruct}) and define the sequence
\begin{equation}\label{eq:exp_alg}
	\begin{cases}
		k_1=0,\\
		k_2:=\alpha, \\
		k_{m+1}:=\alpha(k_{m}+1), \qquad m\geq2, 
	\end{cases}
	\quad \mbox{where}  \quad \alpha:=\displaystyle\frac{p-1}{p}+\frac{1}{\theta}.
\end{equation}
Then, for any $m\in \mathbb N$, there exist constants $\delta_m>0$ and $C>0$ such that if $u\in H^1(a,b)$ satisfies
\begin{equation}\label{|w-v|_L^1<delta_l}
	\|u-v\|_{L^1}\leq\delta_m,
\end{equation}
and
\begin{equation}\label{E_p(w)<Nc0+eps^l}
	E_\e[u]\leq Nc_p+C\varepsilon^{k_{m}},
\end{equation}
with $\varepsilon$ sufficiently small, then
\begin{equation}\label{eq:lower_alg}
	E_\e[u]\geq Nc_p-C_m\varepsilon^{k_{m+1}},
\end{equation}
where $E_\e$ and $c_p$ are defined in \eqref{eq:energy} and \eqref{eq:c_p}, respectively.
\end{prop}

For the proof of this result see \cite[Proposition 4.1]{DCDS}.

\begin{rem}\label{sharp}
First of all, we observe that the assumption $\theta>p$ implies that $\alpha\in(0,1)$ and, consequently, 
the increasing sequence defined in \eqref{eq:exp_alg} satisfies
\begin{equation}\label{eq:limitalpha}
	\lim_{m\to+\infty}k_m=\frac{\alpha}{1-\alpha}=\frac{\theta p}{\theta-p}-1.
\end{equation}
Therefore, the assumption $\theta>p$ implies that the sequence \eqref{eq:exp_alg} is bounded from above and from \eqref{eq:limitalpha} it follows that, for  $\e\to0^+$,
the \emph{best exponent} we can obtain is
\begin{equation*}
	\gamma_{\theta,p}:=\lim_{m\to+\infty}k_m=\frac{\theta p}{\theta-p}-1.
\end{equation*}
\end{rem}

\subsection{Main results}
Lemma \ref{lem:energy-dec} and Propositions \ref{prop:lower}-\ref{prop:lower_deg} are the key ingredients to apply the energy approach 
introduced in \cite{Bron-Hilh,Bron-Kohn,Grant}.
First of all, we consider the case $\theta=p$ and we give the definition of a function with a \emph{transition layer structure}. 
\begin{defn}\label{def:TLS}
We say that a function $u^\e\in H^1(a,b)$ has an \emph{$N$-transition layer structure} if 
\begin{equation}\label{eq:ass-u0}
	\lim_{\varepsilon\rightarrow 0} \|u^\varepsilon-v\|_{{}_{L^1}}=0,
\end{equation}
where $v$ is as in \eqref{vstruct}, and there exist constants $C>0$, $A\in(0,r\sqrt2\lambda_p)$ 
(with $r$ satisfying \eqref{eq:r} and $\lambda_p$ defined in \eqref{eq:lambda}) such that
\begin{equation}\label{eq:energy-ini}
	E_\varepsilon[u^\varepsilon]\leq Nc_p+C\exp(-Ap/2\e),
\end{equation}
for any $\varepsilon\ll1$, where the energy $E_\e$ and the positive constant $c_ p$ are defined in \eqref{eq:energy} and in \eqref{eq:c_p}, respectively.
\end{defn}

Our first result states that the solution $u^\e(\cdot,t)$ arising from an initial datum satisfying \eqref{eq:ass-u0} and \eqref{eq:energy-ini}, 
satisfies the property \eqref{eq:ass-u0}  as well, (at least) for an exponentially long time.
Together with \eqref{eq:energy-est}, this ensures that the solution maintains the same transition layer structure of the initial datum for an exponentially long time,
thus exhibiting a metastable dynamics. It is important to notice that, for $N \geq 2$, profiles with a transition layer structure as the one introduced in Definition \ref{def:TLS}, are neither stationary solutions to \eqref{eq:CH-model}-\eqref{eq:Neu-p} nor they are close to them because of the results of Section \ref{sec:steady}  
(see, in particular, subsection \ref{rem:2trans}, where we proved that stationary solutions with a transition layer structure exist, but the layers are not randomly located).

\begin{thm}[metastable dynamics in the critical case $\theta=p$]\label{thm:main}
Let $v$ be as in \eqref{vstruct} and $A\in(0,r\sqrt{2}\lambda_p)$, with $r$ satisfying \eqref{eq:r} and $\lambda_p$ defined in \eqref{eq:lambda}.
If $u^\varepsilon$ is the solution to \eqref{eq:CH-model}, with $D\in C^1$ strictly positive, $F$ given by \eqref{eq:F} and $\theta=p>1$, 
subject to boundary conditions \eqref{eq:Neu-p} and initial datum $u_0^{\varepsilon}$ satisfying \eqref{eq:ass-u0} and \eqref{eq:energy-ini}, then, 
\begin{equation}\label{eq:limit}
	\sup_{0\leq t\leq m \, {\exp(Ap/2\varepsilon)}}\|\tilde u^\varepsilon(\cdot,t)-\tilde v\|_{{}_{L^1}}\xrightarrow[\varepsilon\rightarrow0]{}0,
\end{equation}
for any $m>0$. 
Moreover, if $p\geq2$, then 
\begin{equation}\label{eq:limit2}
	\sup_{0\leq t\leq m\e^\eta \, {\exp(Ap/2\varepsilon)}}\|u^\varepsilon(\cdot,t)-v\|_{{}_{L^1}}\xrightarrow[\varepsilon\rightarrow0]{}0,
\end{equation}
for any $m>0$ and any $\eta\in(1-2/p,1)$.
\end{thm}

As it was already mentioned, thanks to Lemma \ref{lem:energy-dec} and Proposition \ref{prop:lower}, 
we can apply the same strategy of \cite{Bron-Hilh,Grant} to prove Theorem \ref{thm:main}.
The first step of the proof is the following bound on the $L^2$--norm of the time derivative of the solution $u_t^\varepsilon$.
\begin{prop}\label{prop:L2-norm}
Under the same assumptions of Theorem \ref{thm:main}, there exist positive constants $\varepsilon_0, C_1, C_2>0$ (independent on $\varepsilon$) such that
\begin{equation}\label{L2-norm}
	\int_0^{C_1\varepsilon^{-1}\exp(Ap/2\varepsilon)}\|\tilde u_t^\varepsilon\|^2_{{}_{L^2}}dt\leq C_2\varepsilon\exp(-Ap/2\varepsilon),
	\end{equation}
for all $\varepsilon\in(0,\varepsilon_0)$.
\end{prop}

\begin{proof}
Let $\varepsilon_0>0$ so small that for all $\varepsilon\in(0,\varepsilon_0)$, \eqref{eq:ass-u0} holds and 
\begin{equation}\label{1/2delta}
	\|u_0^\varepsilon-v\|_{{}_{L^1}}\leq\frac12\delta(b-a)^{-1},
\end{equation}
where $\delta$ is the constant of Proposition \ref{prop:lower}. 
From \eqref{1/2delta} and the definition of $\tilde u,\tilde v$, it follows that
\begin{equation}\label{1/2delta-int}
	\|\tilde u_0^\varepsilon-\tilde{v}\|_{{}_{L^1}}\leq\int_a^b\left[\int_a^x|u_0^\e(y)-v(y)|\,dy\right]\,dx\leq(b-a)\|u_0^\varepsilon-v\|_{{}_{L^1}}\leq\frac12\delta.
\end{equation}
Let $T_\e>0$; we claim that if
\begin{equation}\label{claim1}
	\int_0^{T_\e}\|\tilde u_t^\varepsilon\|_{{}_{L^1}}dt\leq\frac12\delta,
\end{equation}
then there exists $C>0$ such that	
\begin{equation}\label{claim2}
	E_\varepsilon[u^\varepsilon](T_\e)\geq Nc_p-C\exp(-Ap/2\varepsilon).
\end{equation}
Indeed, by using \eqref{1/2delta-int}, \eqref{claim1} and the triangle inequality we obtain
\begin{equation*}
	\|\tilde u^\varepsilon(\cdot,T_\e)-\tilde v\|_{{}_{L^1}}\leq\|\tilde u^\varepsilon(\cdot,T_\e)-\tilde u_0^\varepsilon\|_{{}_{L^1}}+\|\tilde u_0^\varepsilon-\tilde v\|_{{}_{L^1}}
	\leq\int_0^{T_\e}\|\tilde u_t^\varepsilon\|_{{}_{L^1}}+\frac12\delta\leq\delta,
\end{equation*}
and inequality \eqref{claim2} follows from Proposition \ref{prop:lower}.
Substituting \eqref{eq:energy-ini} and \eqref{claim2} in \eqref{eq:energy-est} yields
\begin{equation}\label{L2-norm-Teps}
	\int_0^{T_\e}\|\tilde u_t^\varepsilon\|^2_{{}_{L^2}}dt\leq C_2\e\exp(-Ap/2\varepsilon).
\end{equation}
It remains to prove that inequality \eqref{claim1} holds for $T_\e\geq C_1\e^{-1}\exp(Ap/2\varepsilon)$.
If 
\begin{equation*}
	\int_0^{+\infty}\|\tilde u_t^\varepsilon\|_{{}_{L^1}}dt\leq\frac12\delta,
\end{equation*}
then there is nothing to prove. 
Otherwise, choose $\hat T_\e$ such that
\begin{equation*}
	\int_0^{\hat T_\e}\|\tilde u_t^\varepsilon\|_{{}_{L^1}}dt=\frac12\delta.
\end{equation*}
Using H\"older's inequality and \eqref{L2-norm-Teps}, we infer
\begin{equation*}
	\frac12\delta\leq[\hat T_\e(b-a)]^{1/2}\biggl(\int_0^{\hat T_\e}\|\tilde u_t^\varepsilon\|^2_{{}_{L^2}}dt\biggr)^{1/2}\leq
	\left[\hat T_\e(b-a)C_2\varepsilon\exp(-Ap/2\varepsilon)\right]^{1/2}.
\end{equation*}
It follows that there exists $C_1>0$ such that
\begin{equation*}
	\hat T_\e\geq C_1\varepsilon^{-1}\exp(Ap/2\varepsilon),
\end{equation*}
and the proof is complete.
\end{proof}
We now have all the tools to prove \eqref{eq:limit}-\eqref{eq:limit2}.
\begin{proof}[Proof of Theorem \ref{thm:main}]
The triangle inequality gives us 
\begin{equation}\label{trianglebar}
	\|\tilde{u}^\varepsilon(\cdot,t)-\tilde v\|_{{}_{L^1}}\leq\|\tilde{u}^\varepsilon(\cdot,t)-\tilde{u}_0^\varepsilon\|_{{}_{L^1}}+\|\tilde{u}_0^\varepsilon-\tilde v\|_{{}_{L^1}},
\end{equation}
for all $t\in[0, m\exp(Ap/2\varepsilon)]$. 
The last term of inequality \eqref{trianglebar} tends to $0$ by assumption \eqref{eq:ass-u0} and by \eqref{1/2delta-int}; 
let us show that also the first one tends to 0 as $\e\to0$.
To this end, taking $\varepsilon$ small enough so that $C_1\varepsilon^{-1}\geq m$, we can apply Proposition \ref{prop:L2-norm} and, by using \eqref{L2-norm}, we deduce that
\begin{align*}
	\|\tilde{u}^\e(\cdot,t)-\tilde{u}^\e_0\|^2_{{}_{L^1}} & \leq (b-a)\|\tilde{u}^\e(\cdot,t)-\tilde{u}^\e_0\|^2_{{}_{L^2}}=(b-a) \left\|\int_0^t\tilde{u}_s^\e(\cdot,s)\,ds\right\|^2_{L^2}\\
	& \leq (b-a) t\int_0^t\|\tilde{u}_t^\e\|^2_{{}_{L^2}}dt \\
	& \leq (b-a) m\exp(Ap/2\varepsilon)\int_0^{m\exp(Ap/2\varepsilon)}\|\tilde{u}_t^\e\|^2_{{}_{L^2}}dt \\
	& \leq C_2(b-a)m\e,
\end{align*}	
for all $t\in[0,m\exp(Ap/2\varepsilon)]$. Hence, \eqref{eq:limit} follows.

Moreover, fix $\eta\in(1-2/p,1)$, and, as before, 
\begin{equation}\label{triangle}
	\|u^\varepsilon(\cdot,t)- v\|_{{}_{L^1}}\leq\|u^\varepsilon(\cdot,t)-u_0^\varepsilon\|_{{}_{L^1}}+\|u_0^\varepsilon-v\|_{{}_{L^1}},
\end{equation}
for all $t\in[0,m\e^\eta\exp(Ap/2\varepsilon)]$. As before,  from \eqref{eq:ass-u0}, we only need to control the first term on the right hand side of  \eqref{triangle}.
Taking  $\varepsilon$ small enough so that $C_1\varepsilon^{-1}\geq m\e^\eta$,  we thus obtain 
\begin{equation}\label{eq:bar}
	\|\tilde{u}^\e(\cdot,t)-\tilde{u}^\e_0\|^2_{{}_{L^2}}=\left\|\int_0^t\tilde{u}_s^\e(\cdot,s)\,ds\right\|^2_{L^2}\leq t\int_0^t\|\tilde{u}_t^\e\|^2_{{}_{L^2}}dt\leq C_2m\e^{1+\eta},
\end{equation}	
for all $t\in[0,m\e^\eta\exp(Ap/2\varepsilon)]$.

Denote by $w^\e(x,t):=u^\e(x,t)-u^\e_0(x)$ and $\tilde{w}^\e(x,t):=\tilde{u}^\e(x,t)-\tilde{u}^\e_0(x)$.
Integrating by parts and using the boundary conditions \eqref{eq:bound-ut} for $\tilde u$, we infer
\begin{equation}\label{ineq:w}
	\begin{aligned}
		\|w^\e(\cdot,t)\|^2_{{}_{L^2}}&=\int_a^b\tilde{w}_x^\e(x,t)w^\e(x,t)\,dx=-\int_a^b\tilde{w}^\e(x,t)w^\e_x(x,t)\,dx\\
		& \leq\|\tilde{w}^\e(\cdot,t)\|_{{}_{L^2}}\|w^\e_x(\cdot,t)\|_{{}_{L^2}}.
	\end{aligned}
\end{equation}
In order to estimate the last term of \eqref{ineq:w}, we use \eqref{eq:bar}, the assumption \eqref{eq:energy-ini} and \eqref{eq:energy-est}.
Indeed, since $\int_a^b|u^\e_x(x,t)|^pdx\leq C\varepsilon^{1-p}$ for all $t\geq0$, if $p\geq2$, then 
\begin{equation*}
	\|w^\e_x(\cdot,t)\|_{{}_{L^2}}\leq C\|w^\e_x(\cdot,t)\|_{{}_{L^p}}\leq C\e^{\frac{1-p}{p}},
\end{equation*}
and we end up with
\begin{equation*}
	\|u^\e(\cdot,t)-u^\e_0\|^2_{{}_{L^2}}\leq C\e^{\frac{1+\eta}2+\frac{1-p}{p}}=C\e^{\frac{\eta p-p+2}{2p}},
\end{equation*}
for all $t\in[0,m\e^\eta\exp(Ap/2\varepsilon)]$.
Notice that, since $\eta\in(1-2/p,1)$ the exponent of $\e$ is strictly positive.
It finally follows that
\begin{equation*}
	\sup_{0\leq t\leq m\e^\eta\exp(Ap/2\varepsilon)}\|u^\varepsilon(\cdot,t)-u_0^\varepsilon\|_{{}_{L^1}}\leq C\e^{\frac{\eta p-p+2}{4p}}.
\end{equation*}
Combining the latter estimate, \eqref{triangle} and by passing to the limit as $\varepsilon\to0$, we obtain \eqref{eq:limit2}.
\end{proof}

Theorem \ref{thm:main} provides sufficient conditions for the existence of a metastable state for equation \eqref{eq:CH-model}
and shows its persistence for (at least) an exponentially long time.
It is also of interest to notice that the bigger is $p$, the longer the time of such a persistence. 
Also, we recover exactly the classical result when $p=2$ (cfr. \cite{Bates-Xun1,Bates-Xun2}).

We now consider the case $\theta>p$, and we prove the algebraic slow motion of the solutions.
As done before, we fix a piecewise constant function $v$ with $N$ transitions as in \eqref{vstruct} and we assume that the initial datum $u^\e_0$ satisfies
\begin{equation}\label{eq:ass-u0-deg}
	\lim_{\varepsilon\rightarrow 0} \|u^\varepsilon_0-v\|_{{}_{L^1}}=0,
\end{equation}
and that there exist $C>0$ and $m\in\mathbb N$ such that
\begin{equation}\label{eq:energy-ini-deg}
	E_\varepsilon[u^\varepsilon_0]\leq Nc_p+C\e^{k_m},
\end{equation}
for any $\varepsilon\ll1$, where the energy $E_\e$ and the positive constants $c_ p, k_m$ are defined in \eqref{eq:energy}, \eqref{eq:c_p} and \eqref{eq:exp_alg}, respectively.

Our second result is the following theorem.
\begin{thm}[algebraic slow motion in the degenerate or supercritical case $\theta>p$]\label{thm:main2}
Let $u^\varepsilon$ be the solution to \eqref{eq:CH-model}-\eqref{eq:F}-\eqref{eq:initial}-\eqref{eq:Neu-p} with $\theta>p>1$,  
$D\in C^1$ strictly positive and with initial datum $u_0^{\varepsilon}$ satisfying \eqref{eq:ass-u0-deg}-\eqref{eq:energy-ini-deg}. 
Then,
\begin{equation}\label{eq:limit-deg}
	\sup_{0\leq t\leq l{\e^{-k_m}}}\|\tilde u^\varepsilon(\cdot,t)-\tilde v\|_{{}_{L^1}}\xrightarrow[\varepsilon\rightarrow0]{}0,
\end{equation}
for any $l>0$.
Moreover, if $p\geq2$, then 
\begin{equation}\label{eq:limit-deg2}
	\sup_{0\leq t\leq l\e^{-k_m+\eta}}\|u^\varepsilon(\cdot,t)-v\|_{{}_{L^1}}\xrightarrow[\varepsilon\rightarrow0]{}0,
\end{equation}
for any $l>0$ and any $\eta\in(1-2/p,1)$.
\end{thm}

The strategy to prove Theorem \ref{thm:main2} is the same of Theorem \ref{thm:main}, but with the crucial difference that we need to use Proposition \ref{prop:lower_deg};
to do this, we need to verify assumption \eqref{|w-v|_L^1<delta_l} at a large time $T_\e>0$ and this complicates the computations, as we show in the proof
of the following instrumental result, that plays the same role of Proposition \ref{prop:L2-norm}.

\begin{prop}\label{prop:L2-norm-alg} 
Let $\mu\in(0,1)$. 
Under the same assumptions of Theorem \ref{thm:main2}, there exist positive constants $\varepsilon_0, C_1, C_2>0$ (independent on $\varepsilon$) such that
\begin{equation}\label{L2-norm-alg}
	\int_0^{C_1\varepsilon^{-(k_m+\mu)}}\|\tilde u_t^\varepsilon\|^2_{{}_{L^2}}dt\leq C_2\varepsilon^{k_m+1},
	\end{equation}
for all $\varepsilon\in(0,\varepsilon_0)$.
\end{prop}

\begin{proof}
First of all, notice that proceeding as in \eqref{1/2delta-int}, we get
\begin{equation*}
	\|\tilde u_0^\varepsilon-\tilde{v}\|_{{}_{L^1}}\leq(b-a)\|u_0^\varepsilon-v\|_{{}_{L^1}},
\end{equation*}
and, as a consequence, the assumption \eqref{eq:ass-u0-deg} ensures 
\begin{equation}\label{eq:tilde-u-v}
	\lim_{\e\to0}\|\tilde u_0^\varepsilon-\tilde{v}\|_{{}_{L^1}}=0.
\end{equation}
Similarly to the proof of Proposition \ref{prop:L2-norm}, we proceed in two steps: 
first, we claim that if $T_\e>0$ satisfies
\begin{equation}\label{eq:claim1-alg}
	\int_0^{T_\e}\|\tilde u_t^\varepsilon\|_{{}_{L^1}}dt\leq\e^{\alpha}, \qquad \qquad \alpha=\frac12(1-\mu)>0,
\end{equation}
then, there exists $C>0$ such that	
\begin{equation}\label{eq:L2-norm-T-alg}
	\int_0^{T_\e}\|\tilde u_t^\varepsilon\|^2_{{}_{L^2}}dt\leq C\e^{k_m+1}.
\end{equation}
Then, we prove that \eqref{eq:L2-norm-T-alg} holds true for $T_\e\geq C_1\e^{-(km+\mu)}$.
The fact that \eqref{eq:claim1-alg} implies \eqref{eq:L2-norm-T-alg} is a consequence of the energy estimate \eqref{eq:energy-est} and the lower bound \eqref{eq:lower_alg}.
Thus, let us prove that \eqref{eq:claim1-alg} implies that the solution at time $T_\e$, i.e. $u^\e(\cdot,T_\e)$, satisfies the assumptions of Proposition \ref{prop:lower_deg}. 
Since the energy does not increase in time along the solution, see \eqref{eq:energy-est}, and the initial datum satisfies \eqref{eq:energy-ini-deg},
the solution $u^\e(\cdot,t)$ verifies assumption \eqref{E_p(w)<Nc0+eps^l} for any time $t\geq0$ and we only have to prove that \eqref{|w-v|_L^1<delta_l} holds true at time $t=T_\e$.
To do this, we first notice that, since
\begin{equation*}
	\|\tilde u^\varepsilon(\cdot,T_\e)-\tilde v\|_{{}_{L^1}}\leq\|\tilde u^\varepsilon(\cdot,T_\e)-\tilde u_0^\varepsilon\|_{{}_{L^1}}+\|\tilde u_0^\varepsilon-\tilde v\|_{{}_{L^1}}
	\leq\int_0^{T_\e}\|\tilde u_t^\varepsilon\|_{{}_{L^1}}+\|\tilde u_0^\varepsilon-\tilde v\|_{{}_{L^1}},
\end{equation*}
it follows from \eqref{eq:tilde-u-v} and \eqref{eq:claim1-alg} that
\begin{equation}\label{eq:tilde u_hat T}
	\lim_{\e\to0}\|\tilde u^\varepsilon(\cdot,T_\e)-\tilde v\|_{{}_{L^1}}=0.
\end{equation}
Next, Lemma \ref{lem:energy-dec} and assumption \eqref{eq:energy-ini-deg} give $E_\e[u^\e](t)\leq C$, for any $t\geq0$, that is
\begin{equation}\label{eq:crucial}
	\int_a^b\left[\frac{\e^{p-1}|u^\e_x|^p}{p}+\frac{F(u^\e)}\e\right]\,dx\leq C, \qquad \forall\,t\geq0.
\end{equation}
Therefore, we can prove that the function $\Psi(u^\e(\cdot,T_\e))$ is uniformly bounded in $BV(a,b)$, where
\begin{equation*}
	\Psi(z):=\int_{0}^{z} F(s)^{\frac{p-1}{p}}\,ds=\left(\frac{1}{2\theta}\right)^{\frac{p-1}{p}}\int_{0}^{z} |1-s^2|^{\theta-\theta/p}\,ds.
\end{equation*}
Indeed, applying H\"older's inequality and \eqref{eq:crucial}, we obtain
\begin{align*}
	\int_a^b\left|\frac{d}{dx}\Psi(u^\e(x,T_\e))\right|\,dx&=\int_a^b|\Psi'(u^\e(x,T_\e))||u^\e_x(x,T_\e)|\,dx\\
	&=\int_a^b F(u^\e(x,T_\e))^{\frac{p-1}{p}}|u_x^\e(x,T_\e)|\,dx\\
	&\leq\|\e u^\e_x(\cdot,T_\e)\|_{{}_{L^p}}+\|\e^{-1}F(u^\e(\cdot,T_\e))\|^\frac{p}{p-1}_{{}_{L^1}}\leq C.
\end{align*}
Moreover, since $\Psi(z)\approx z^{2\theta-2\theta/p+1}$ for large $z$ and $2\theta-2\theta/p+1<2\theta$, because $\theta>p$, 
we can choose constants $c_1,c_2$ such that
$$|\Psi(z)|\leq c_1+c_2|1-z^2|^{\theta},$$
and, as a consequence, 
$$\int_a^b\left|\Psi(u^\e(x,T_\e))\right|\,dx\leq c_1(b-a)+2c_2\theta\int_a^b F(u^\e(x,T_\e))\,dx\leq c_1(b-a)+2c_2\theta C\e, $$
where we used again \eqref{eq:crucial}.
Therefore, $\Psi(u^\e(\cdot,T_\e))$ is uniformly bounded in $BV(a,b)$ and for a standard compactness result, we can state that
there exists a subsequence $\Psi(u^{\e_j}(\cdot,T_{\e_j}))$ which converges in $L^1(a,b)$ to a function $\Psi^*$, namely 
\begin{equation}\label{eq:Psi*}
	\lim_{\e_j\to0}\|\Psi(u^{\e_j}(\cdot,T_{\e_j}))-\Psi^*\|_{{}_{L^1}}=0.
\end{equation}
Passing to a further subsequence if necessary, we obtain
\begin{equation*}
	\lim_{\e_j\to0}\Psi(u^{\e_j}(x,T_{\e_j}))=\Psi^*(x), \qquad \qquad  \mbox{ a.e. on }\, (a,b).
\end{equation*}
Since $\Psi'$ is strictly positive except at $\pm1$, the function $\Psi$ is monotone and there is
a unique function $v^*$ such that $\Psi(v^*(x))=\Psi^*(x)$, implying
\begin{equation*}
	\lim_{\e_j\to0} u^{\e_j}(x,T_{\e_j})=v^*(x), \qquad \quad \mbox{ a.e. on }\, (a,b).
\end{equation*}
Using the Fatou's Lemma and \eqref{eq:crucial}, we get
\begin{equation*}
	\int_a^b F(v^*(x,T_{\e_j}))\,dxdt\leq\liminf_{\e_j\to0}\int_a^b F(u^{\e_j}(x,T_{\e_j}))\,dxdt=0,
\end{equation*}
so that $v^*$ takes only the values $\pm1$.
The latter property, together with \eqref{eq:tilde u_hat T}, imply that $v^*=v$ a.e. on $(a,b)$ or, equivalently, that
\begin{equation*}
	\lim_{\e_j\to0} u^{\e_j}(x,T_{\e_j})=v(x), \qquad \quad \mbox{ a.e. on }\, (a,b).
\end{equation*}
To prove this, it is sufficient to proceed as in \eqref{eq:keyidea} and use \eqref{eq:tilde u_hat T} to show that
$$\left|\int_a^b(u^{\e_j}(x,T_{\e_j})-v(x))w(x)\,dx\right|\leq\|w'\|_{{}_{L^\infty}}\|\tilde u^\varepsilon(\cdot,T_\e)-\tilde v\|_{{}_{L^1}}\rightarrow0, \qquad \mbox{ as } \, \e\to0,$$
for any test function $w\in C^1_c([a,b])$.
Hence, if we suppose by contradiction that $v\neq v^*$ in some set $I\subset(a,b)$ with $meas(I)>0$, 
we obtain a contradiction because $|u^{\e_j}(\cdot,T_{\e_j})-v|>1$ in $I$.
The last step is to prove that $u^{\e_j}(\cdot,T_{\e_j})$ converges to $v$ in $L^1(a,b)$.
From \eqref{eq:Psi*} and by the  strict monotonicity of $\Psi$, it follows that
\begin{equation*}
	\lim_{\e\to0}\|\Psi(u^{\e}(\cdot,T_{\e}))-\Psi(v)\|_{{}_{L^1}}=0.
\end{equation*}
However, using again that $\Psi(z)\approx z^{2\theta-2\theta/p+1}$ for large $z$, we finally deduce
\begin{equation*}
	\lim_{\e\to0}\|u^{\e}(\cdot,T_{\e})-v\|_{{}_{L^1}}=0.
\end{equation*}
As a consequence, we can choose $\e$ so small that $u^{\e}(\cdot,T_{\e})$ satisfies assumption \eqref{|w-v|_L^1<delta_l}, 
and applying Proposition \ref{prop:lower_deg}, we have 
$$E[u^\e](T_\e)\geq Nc_p-C_m\e^{k_{m+1}}.$$
Furthermore, by using \eqref{eq:energy-est} and \eqref{eq:energy-ini-deg} we obtain
$$\int_0^{T_\e}\|\tilde u_t^\varepsilon\|^2_{{}_{L^2}}dt\leq C\e\left(E_\e[u^\e_0]-E[u^\e](T_\e)\right)\leq C\e^{k_m+1}.$$
We proved that \eqref{eq:claim1-alg} implies \eqref{eq:L2-norm-T-alg}; 
it remains to prove that inequality \eqref{eq:claim1-alg} holds for $T_\e\geq C_1\e^{-(k_m+\mu)}$.
If 
\begin{equation*}
	\int_0^{+\infty}\|\tilde u_t^\varepsilon\|_{{}_{L^1}}dt\leq\e^{\alpha},
\end{equation*}
then there is nothing to prove. 
Otherwise, choose $\hat T_\e$ such that
\begin{equation*}
	\int_0^{\hat T_\e}\|\tilde u_t^\varepsilon\|_{{}_{L^1}}dt=\e^\alpha.
\end{equation*}
Using H\"older's inequality and \eqref{eq:L2-norm-T-alg}, we infer
\begin{equation*}
	\e^\alpha\leq[\hat T_\e(b-a)]^{1/2}\biggl(\int_0^{\hat T}\|\tilde u_t^\varepsilon\|^2_{{}_{L^2}}dt\biggr)^{1/2}\leq
	\left[\hat T_\e(b-a)C_2\varepsilon^{k_m+1}\right]^{1/2}.
\end{equation*}
It follows that there exists $C_1>0$ such that
\begin{equation*}
	\hat T_\e\geq C_1\varepsilon^{2\alpha-k_m-1}.
\end{equation*}
Hence, by using $\alpha=\frac12(1-\mu)$, we deduce that
\begin{equation*}
	\hat T_\e\geq C_1\varepsilon^{-(k_m+\mu)}.
\end{equation*}
and the proof is complete.
\end{proof}

We now have all the tools to prove \eqref{eq:limit-deg}-\eqref{eq:limit-deg2}, proceeding in the same way as we have done for \eqref{eq:limit}-\eqref{eq:limit2}.
\begin{proof}[Proof of Theorem \ref{thm:main2}]
Let us proceed as in the proof of Theorem \ref{thm:main};
triangle inequality gives us \eqref{trianglebar} for all $t\in[0,l\varepsilon^{-k_m}]$, 
and the last term of inequality \eqref{trianglebar} tends to $0$ by assumption \eqref{eq:tilde-u-v}.
Taking $\varepsilon$ small enough so that $C_1\varepsilon^{-\delta}\geq l$, we can apply Proposition \ref{prop:L2-norm-alg} and, by using \eqref{L2-norm-alg}, we deduce
\begin{align*}
	\|\tilde{u}^\e(\cdot,t)-\tilde{u}^\e_0\|^2_{{}_{L^1}} & \leq (b-a)\|\tilde{u}^\e(\cdot,t)-\tilde{u}^\e_0\|^2_{{}_{L^2}}\leq (b-a) t\int_0^t\|\tilde{u}_t^\e\|^2_{{}_{L^2}}dt \\
	& \leq (b-a) l\varepsilon^{-k_m}\int_0^{l\varepsilon^{-k_m}}\|\tilde{u}_t^\e\|^2_{{}_{L^2}}dt \\
	& \leq C_2(b-a)l\e,
\end{align*}	
for all $t\in[0,l\varepsilon^{-k_m}]$ and \eqref{eq:limit} follows.

The proof of \eqref{eq:limit-deg2} is very similar to the one of \eqref{eq:limit2}, for details see the proof of Theorem \ref{thm:main}.
\end{proof}

We conclude this section by recalling that \cite[Section 3.1]{DCDS} ensures the existence of a family of functions $u^\e$ with a transition layer structure,
that is a family of functions satisfying the assumptions \eqref{eq:ass-u0}, \eqref{eq:energy-ini} or \eqref{eq:energy-ini-deg}, 
that we required in the main results, Theorems \ref{thm:main} and \ref{thm:main2}. The proof consists in explicitly constructing the function $u^\e$.

\section{Layer dynamics}\label{LD}
In this last section we aim at showing how the results of Section \ref{sec:slow} can be translated into a result concerning the motion of the transition points $h_1, \dots , h_N$. 
Theorems \ref{thm:main} and \ref{thm:main2} show that solutions to \eqref{eq:CH-model}-\eqref{eq:Neu-p}  arising from initial data with $N$-transition layers maintain such unstable structure for long times (precisely, exponentially long times or algebraically long times if $\theta=p$ or $\theta>p$, respectively). 
These results are tantamount to a precise description of the motion of the transition points $h_1, \dots , h_N$, showing that they move with a very small velocity as $\e\to0^+$.

Following the strategy of \cite{MMAS19, DCDS, Grant}, let us consider $v$ a piecewise constance function as in \eqref{vstruct}, and $u: [a,b] \to \mathbb{R}$ an arbitrary function. 
We define their {\it interfaces} as follows:
\begin{equation*}
	I[v]=\{h_1, \dots , h_N\} \qquad \mbox{and} \qquad I_K[u]=u^{-1}(K),
\end{equation*}
where $K \subset \mathbb{R}\setminus \{ \pm 1\}$ is an arbitrary closed subset. Also, for any $A,B \subset \mathbb{R}$, we define 
\begin{equation*}
	d(A,B):=\max \left \{ \sup_{\alpha \in A} d(\alpha,B), \sup_{\beta \in B} d(\beta,A)\right \},
\end{equation*}
where $d(\beta,A):=\inf \{|\beta-\alpha| \, : \,  \alpha \in A \}$.

The next  lemma  shows that  the distance between the interfaces $I_K[u]$
and $I[v]$ is small, providing some  smallness assumptions on the $L^1$--norm of the difference $ u-v$ and on the energy $E_\e[u].$
The result is purely variational in character and holds true both in the critical ($\theta=p$) and supercritical ($\theta>p$) cases.

\begin{lem}\label{lemma:layermotion}
Let $F$ as in \eqref{eq:F}, $v$ as in \eqref{vstruct} and $r$ such that \eqref{eq:r} holds. 
Given $\delta \in (0,r)$, there exist constants $\hat \delta, \e_0, \Gamma >0$ such that, if $u \in H^1([a,b])$ satisfies
\begin{equation}\label{ipolemma}
	\|\tilde u-\tilde v\|_{L^1} < \hat \delta \qquad \mbox{and} \qquad  E_\e[u] \leq Nc_p + \Gamma,
\end{equation}
then, for any $\e \in (0,\e_0)$ there holds
\begin{equation}\label{tesilemma}
	d(I_K[u],I[v])\leq \frac{\delta}{2}.
\end{equation}
\end{lem}

\begin{proof}
Let us fix $\delta\in(0,r)$ and choose $\rho>0$ small enough that 
\begin{equation*}
	I_\rho:=(-1-\rho,-1+\rho)\cup(1-\rho,1+\rho)\subset\R\backslash K, 
\end{equation*}
and 
\begin{equation*}
	\inf\left\{\left(\frac{p}{p-1}\right)^{\frac{p-1}{p}}\left|\int_{\xi_1}^{\xi_2} {F(s)}^\frac{p-1}{p}\, ds\right| : \xi_1\in K, \xi_2\in I_\rho\right\}>2M,
\end{equation*}
where
\begin{equation*}
	M:=2N \left(\frac{p}{p-1}\right)^{\frac{p-1}{p}}\max\left\{\int_{1-\rho}^{1}{F(s)}^\frac{p-1}{p}\\,ds, \, \int_{-1}^{-1+\rho}{F(s)}^\frac{p-1}{p}\\,ds \right\}.
\end{equation*}
By reasoning as in the proof of \eqref{2points} in Proposition \ref{prop:lower}, we can prove that for each $i$ there exist
\begin{equation*}
	x^-_{i}\in(h_i-\delta/2,h_i) \qquad \textrm{and} \qquad x^+_{i}\in(h_i,h_i+\delta/2),
\end{equation*}
such that
\begin{equation*}
	|u(x^-_{i})-v(x^-_{i})|<\rho \qquad \textrm{and} \qquad |u(x^+_{i})-v(x^+_{i})|<\rho.
\end{equation*}
	Suppose by absurd  that \eqref{tesilemma} is violated, and let's show that this leads to a contradiction. By Young's inequality we deduce
\begin{align}
	{E}_\varepsilon[u]&\geq\left(\frac{p}{p-1}\right)^{\frac{p-1}{p}}\sum_{i=1}^N\left|\int_{u(x^-_{i})}^{u(x^+_{i})}{F(s)}^\frac{p-1}{p}\,ds\right|\notag\\ 
	&  \quad+\left(\frac{p}{p-1}\right)^{\frac{p-1}{p}}\inf\left\{\left|\int_{\xi_1}^{\xi_2}{F(s)}^\frac{p-1}{p}\,ds\right| : \xi_1\in K, \xi_2\in I_\rho\right\}. \label{diseq:E1}
\end{align}
On the other hand, we have
\begin{align*}
	\left(\frac{p}{p-1}\right)^{\frac{p-1}{p}}\left|\int_{u(x^-_{i})}^{u(x^+_{i})}{F(s)}^\frac{p-1}{p}\,ds\right|&\geq\left(\frac{p}{p-1}\right)^{\frac{p-1}{p}}\int_{-1}^{1}{F(s)}^\frac{p-1}{p}\,ds\\			&\qquad-\left(\frac{p}{p-1}\right)^{\frac{p-1}{p}}\int_{-1}^{-1+\rho}{F(s)}^\frac{p-1}{p}\,ds\\
	&\qquad -\left(\frac{p}{p-1}\right)^{\frac{p-1}{p}}\int_{1-\rho}^{1}{F(s)}^\frac{p-1}{p}\,ds\\
	&\geq c_p-\frac{M}{N}. 
\end{align*}
Substituting the latter bound in \eqref{diseq:E1}, we deduce
\begin{equation*}
	E_\varepsilon[u]\geq N c_p-M+\left(\frac{p}{p-1}\right)^{\frac{p-1}{p}}\inf\left\{\left|\int_{\xi_1}^{\xi_2}{F(s)}^\frac{p-1}{p}\,ds\right| : \xi_1\in K, \xi_2\in I_\rho\right\},
\end{equation*}
that leads, given the choice of $\rho$,  to	
\begin{align*}
	E_\varepsilon[u]>Nc_p+M,
\end{align*}
which is a contradiction with assumption \eqref{ipolemma}. 
\end{proof}

We are now ready to prove the following result concerning the slow motion of the transition points $h_1, \dots , h_N$, showing that  
they evolve exponentially or algebraically slowly if $\theta=p$ or $\theta>p$, respectively.

\begin{thm}\label{thm:main3}
Let $F$ be as in \eqref{eq:F}, $D\in C^1$ strictly positive and let $u^\e$ be the solution to  \eqref{eq:CH-model}-\eqref{eq:initial}-\eqref{eq:Neu-p} 
with initial datum $u^\e_0$ satisfying $\displaystyle{\lim_{\varepsilon\rightarrow 0} \|u^\varepsilon_0-v\|_{{}_{L^1}}=0}$ 
and \eqref{eq:energy-ini} in the case $\theta=p$ or \eqref{eq:energy-ini-deg} in the case $\theta>p$. 
Given $\delta \in (0,r)$, set
\begin{equation*}
	t_\e(\delta) = \inf \{ t  :  d(I_K[u_\e(\cdot, t)], I_K[u_0^\e]) > \delta \},
\end{equation*}
where $K \subset \mathbb{R} \setminus \{ \pm 1\}$. Then there exists $\e_0>0$ such that, if $\e \in (0,\e_0)$
\begin{equation*}
	t_\e(\delta) > \omega(\e),
\end{equation*}
where
\begin{equation*}
	\omega(\e):= \left\{ \begin{aligned}
	&\exp (Ap/2\e) \qquad &\mbox{if} \quad \theta=p, \\
	&\e^{-k_m},  \qquad &\mbox{if} \quad \theta>p,
	\end{aligned}\right.
\end{equation*}
with $A$ and $k_m$ appearing in Theorems \ref{thm:main} and \ref{thm:main2}.
\end{thm}

\begin{proof}
We start with the case $\theta=p$. 
We choose $\e_0$ small enough such that the assumption on $u_0^\e$ implies that \eqref{ipolemma} is satisfied; 
hence, from Lemma \ref{lemma:layermotion} it follows that
\begin{equation*}
	d(I_K[u_0^\e],I[v]) < \frac{\delta}{2}.
\end{equation*}
Also, if considering the time dependent solution $u^\e(\cdot, t)$, from \eqref{eq:limit} in Theorem \ref{thm:main} and since $E_\e[u]$ is a non-increasing function of $t$, 
it follows that \eqref{ipolemma} is satisfied for $u^\e(\cdot, t)$, for any $t < \exp (Ap/2\e)$, implying \eqref{tesilemma} holds for $u^\e$ as well. 
As a consequence, from the triangular inequality, we have
\begin{equation*}
	d(I_K[u^\e(t)],I_K[u_0^\e]) <\delta,
\end{equation*}
for all $t \in (0,\exp(Ap/2\e))$. 
	
When $\theta>p$ we can proceed with the exact same computations by making use of \eqref{eq:limit-deg} in Theorem \ref{thm:main2}; 
we thus end up with 
\begin{equation*}
	d(I_K[u^\e(t)],I_K[u_0^\e]) <\delta,
\end{equation*}
for all $t \in (0,\e^{-k_m})$, and the proof is complete.
\end{proof}

Theorem \ref{thm:main3}, together with Theorems \ref{thm:main} and \ref{thm:main2}, prove that solutions to \eqref{eq:CH-model}-\eqref{eq:Neu-p} with a transition layer structure evolve {\it exponentially slowly} in the case $\theta=p$ and {\it algebraically slowly} if $\theta>p$; they indeed maintain the same profile of their initial datum for times of  
$\mathcal O(\exp(Ap/2\e))$ and $\mathcal O(\e^{-k_m})$ respectively, and the transition points move with exponentially (algebraically respectively) small speed.


\begin{thebibliography}{99}

\bibitem{AlikBateFusc91}
N. D. Alikakos, P. W. Bates and G. Fusco.
Slow motion for the Cahn--Hilliard equation in one space dimension.
{\it J. Differential Equations}, {\bf 90} (1991), 81--135.

\bibitem{Bates-Xun1}
P. W. Bates and J. Xun.
Metastable patterns for the Cahn--Hilliard equation: Part I.
{\it J. Differential Equations}, {\bf 111} (1994), 421--457.

\bibitem{Bates-Xun2}
P. W. Bates and J. Xun.
Metastable patterns for the Cahn--Hilliard equation: Part II. Layer dynamics and slow invariant manifold.
{\it J. Differential Equations}, {\bf 117} (1995), 165--216.

\bibitem{Ben et al}
J. Benedikt, P. Girg, L. Kotrla and P. Tak\'{a}\v{c}.
Origin of the $p$-Laplacian and A. Missbach.
{\it Electron. J. Differ. Eq.}, {\bf2018} (2018), 1--17.

\bibitem{Broad}
P. Broadbridge.
Exact solvability of the Mullins nonlinear diffusion model of groove development. 
{\it J. Math. Phys.}, {\bf30} (1989), 1648--1651.

\bibitem{Bron-Hilh}
L. Bronsard and D. Hilhorst.
On the slow dynamics for the Cahn--Hilliard equation in one space dimension.
{\it Proc. Roy. Soc. London, A}, {\bf439} (1992), 669--682.

\bibitem{Bron-Kohn}
L. Bronsard and R. Kohn.
On the slowness of phase boundary motion in one space dimension.
{\it Comm. Pure Appl. Math.}, {\bf 43} (1990), 983--997.

\bibitem{Cahn}
J. W. Cahn.
On spinodal decomposition.
{\it Acta Metall.}, {\bf9} (1961), 795--801.

\bibitem{Ca-El-NC}
J. W. Cahn, C. M. Elliott and A. Novick-Cohen.
The Cahn--Hilliard equation with a concentration dependent mobility: motion by minus the Laplacian of the mean curvature. 
{\it Eur. J. of Appl. Math.}, {\bf 7} (1996), 287--­301

\bibitem{Cahn-Hilliard} 
J. W. Cahn and J. E. Hilliard. 
Free energy of a nonuniform system. I. Interfacial free energy. 
{\it J. Chem. Phys.}, {\bf 28} (1958), 258--267.

\bibitem{CarrGurtSlem}
J. Carr, M. E. Gurtin, and M. Slemrod.
Structure phase transitions on a finite interval.
{\it Arch. Rat. Mech. Anal.}, {\bf86} (1984),  317--351.

\bibitem{Carr-Pego}
J. Carr and R. L. Pego.
Metastable patterns in solutions of $u_t=\varepsilon^2u_{xx}-f(u)$.
{\it Comm. Pure Appl. Math.}, {\bf 42} (1989), 523--576.

\bibitem{DPV20}
S. Dipierro, A. Pinamonti and E. Valdinoci.
Rigidity results for elliptic boundary value problems: stable solutions for quasilinear equations with Neumann or Robin boundary conditions.
{\it Int. Math. Res. Not. IMRN}, {\bf 2020} (2020), 1366--1384.

\bibitem{Dr-Ma-Ta}
P. Dr\'{a}bek, R. F. Man\'{a}sevich and P. Tak\'{a}\v{c}.
Manifolds of critical points in a quasilinear model for phase transitions. 
{\it Nonlinear elliptic partial differential equations, Contemp. Math.}, {\bf540}, 95--134.

\bibitem{Dra-Rob}
P. Dr\'{a}bek and S. Robinson.
Continua of local minimizers in a quasilinear model of phase transitions.
\emph{Discrete Contin. Dyn. Syst.}, {\bf33} (2013), 163--172.

\bibitem{Dra-Rob2}
P. Dr\'{a}bek and S. Robinson.
Convergence to higher-energy stationary solutions of a bistable equation with non-smooth reaction term.
\emph{Z. Angew. Math. Phys.}, {\bf68} (2017), Paper No. 67, 19 pp.

\bibitem{Fife}
Fife PC.
Models for phase separation and their mathematics.
{\it Electron. J. Differential Equations}, {\bf2000} 2000, 1--26.

\bibitem{MMAS19}
R. Folino, C. Lattanzio and C. Mascia. 
Slow dynamics for the hyperbolic Cahn--Hilliard equation in one-space dimension.
{\it Math. Meth. Appl. Sci.}, {\bf42} (2019), 2492--2512.

\bibitem{JDDE2021}
R. Folino, C. Lattanzio and C. Mascia. 
Metastability and layer dynamics for the hyperbolic relaxation of the Cahn--Hilliard equation.
{\it J. Dynam. Differential Equations}, {\bf33} (2021), 75--110. 

\bibitem{DCDS}
R. Folino, R. G. Plaza and M. Strani. 
Long time dynamics of solutions to $p$-Laplacian diffusion problems with bistable reaction terms.
\emph{Discrete Contin. Dyn. Syst.}, {\bf41} (2021), 3211--3240.

\bibitem{Fusco-Hale}
G. Fusco and J. Hale.
Slow-motion manifolds, dormant instability, and singular perturbations.
{\it J. Dynam. Differential Equations}, {\bf 1} (1989), 75--94.

\bibitem{Grant}
C. P. Grant.
Slow motion in one-dimensional Cahn--Morral systems.
{\it SIAM J. Math. Anal.}, {\bf 26} (1995), 21--34.

\bibitem{Grin-Nov}
M. Grinfeld and A. Novick-Cohen.
Counting stationary solutions of the Cahn-.Hilliard equation by transversality arguments.
{\it Proc. Roy. Soc. Edinburgh Sect. A}, {\bf125} (1995), 351--370.

\bibitem{Gurtin}
M. E. Gurtin.
Generalized Ginzburg--Landau and Cahn--Hilliard equations based on a microforce balance.
{\it Phys. D}, {\bf 92} (1996), 178--192.

\bibitem{KosMorYot}
S. Kosugi, Y. Morita and S. Yotsutani.
Stationary solutions to the one-dimensional Cahn--Hilliard equation: proof by the complete elliptic integrals.
{\it Discrete Contin. Dyn. Syst.}, {\bf19} (2007), 609--629.

\bibitem{Mullins}
W. W. Mullins.
Theory of thermal grooving. 
{\it J. Appl. Phys.}, {\bf28} (1957), 333--339.

\bibitem{Nov-Pel}
A. Novick-Cohen and L. A. Peletier.
Steady states of the one-dimensional Cahn--Hilliard equation.
{\it Proc. Roy. Soc. Edinburgh Sect. A}, {\bf123} (1993), 1071--1098.

\bibitem{Takac}
P. Tak\'{a}\v{c}.
Stationary radial solutions for a quasilinear Cahn--Hilliard model in $N$ space dimensions. 
\emph{Proceedings of the Seventh Mississippi State-UAB Conference on Differential Equations and Computational Simulations, Electron. J. Differ. Equ. Conf.},
{\bf17} (2009), 227--254. 

\bibitem{Wagner}
C. Wagner.
On the solution of diffusion problems involving concentration-dependent diffusion coefficients. 
{\it J. Met.}, {\bf4} (1952), 91--96.

\bibitem{Zheng}
S. Zheng. 
Asymptotic behavior of solutions to the Cahn--Hilliard equation. 
{\it Appl. Anal.}, {\bf23} (1986), 165--184.

\end{thebibliography}
\end{document}